\documentclass[12pt,a4paper]{amsart}

\PassOptionsToPackage{fleqn}{amsmath}

\usepackage[T1]{fontenc}
\usepackage[utf8]{inputenc} 
\usepackage[english]{babel}
\usepackage{lmodern}

\DeclareFontFamily{U}{rsfs}{}
\DeclareFontShape{U}{rsfs}{m}{n}{<-> s*[1.0] rsfs10}{}

\makeatletter
\renewcommand\normalsize{%
    \@setfontsize\normalsize{9.7}{14pt plus .3pt minus .3pt}%
    \abovedisplayskip 10\p@ \@plus4\p@ \@minus4\p@
    \abovedisplayshortskip 6\p@ \@plus2\p@
    \belowdisplayshortskip 6\p@ \@plus2\p@
    \belowdisplayskip \abovedisplayskip}
\renewcommand\small{%
    \@setfontsize\small{9.5}{12\p@ plus .2\p@ minus .2\p@}%
    \abovedisplayskip 8.5\p@ \@plus4\p@ \@minus1\p@
    \belowdisplayskip \abovedisplayskip
    \abovedisplayshortskip \abovedisplayskip
    \belowdisplayshortskip \abovedisplayskip}
\renewcommand\footnotesize{%
    \@setfontsize\footnotesize{8.5}{9.25\p@ plus .1pt minus .1pt}
    \abovedisplayskip 6\p@ \@plus4\p@ \@minus1\p@
    \belowdisplayskip \abovedisplayskip
    \abovedisplayshortskip \abovedisplayskip
    \belowdisplayshortskip \abovedisplayskip}
\ifdefined\pdfpagewidth
\else
\fi
\calclayout
\makeatother

\usepackage{amsmath,amssymb,amsthm}
\usepackage{mathtools}
\usepackage{mathrsfs}
\usepackage{tensor}
\usepackage{braket}
\usepackage{empheq}

\usepackage{graphicx}
\usepackage{tikz}
\usepackage{tikz-cd}
\usepackage{pgfplots}
\pgfplotsset{compat=1.15}
\usetikzlibrary{arrows,positioning,quotes}
\usepackage[all]{xy}

\usepackage[dvipsnames]{xcolor}

\usepackage{enumitem}
\usepackage{comment}
\usepackage[toc,page]{appendix}
\usepackage{chngcntr}
\usepackage{etoolbox}
\usepackage{sidecap}
\usepackage{fullpage}

\definecolor{cobalt}{RGB}{0,71,171}

\usepackage[
  colorlinks,
  bookmarks,
  backref=page
]{hyperref}

\hypersetup{
  colorlinks=true,
  citecolor=cobalt,
  filecolor=black,
  linkcolor=cobalt,
  urlcolor=purple
}

\theoremstyle{plain}
\newtheorem{thm}{Theorem}[section]
\newtheorem{lem}[thm]{Lemma}
\newtheorem{prop}[thm]{Proposition}
\newtheorem{cor}[thm]{Corollary}

\newtheorem{mainthm}{Theorem}

\newtheorem{maincor}[mainthm]{Corollary}

\theoremstyle{definition}
\newtheorem{df}[thm]{Definition}

\newtheorem{mainquest}{Question}

\theoremstyle{remark}
\newtheorem{example}[thm]{Example}
\newtheorem{construction}[thm]{Construction}
\newtheorem{rmk}[thm]{Remark}
\newtheorem{rmkn}[thm]{Remarkman}

\makeatletter
\newcommand{\mylabel}[2]{#2\def\@currentlabel{#2}\label{#1}}
\makeatother

\def\id{\mathrm{id}}
\def\td{\operatorname{td}}

\def\Pic{\mathrm{Pic}}

\def\ch{\mathrm{ch}}

\def\End{\mathrm{End}}

\def\Coh{\mathrm{Coh}}
\def\Z{\mathbb{Z}}

\def\SL{\mathrm{SL}}
\def\H{H}

\def\PGL{\operatorname{PGL}}

\def\Hom{\mathrm{Hom}}
\def\kO{\mathcal{O}}

\def\PP{\mathbb{P}}
\def\pr{\mathrm{pr}}
\def\QQ{\mathbb{Q}}
\def\Mf{\mathfrak{M}}
\def\U{\mathcal{U}}
\def\ext{\mathrm{ext}}

\newcommand{\RR}{\mathbb{R}}
\newcommand{\CC}{\mathbb{C}}

\newcommand{\D}{\mathrm{D}^b}
\newcommand{\kA}{\mathcal{A}}

\newcommand{\kE}{\mathcal{E}}

\newcommand{\ZZ}{\mathbb{Z}}
\newcommand{\rk}{\mathrm{rk}}

\newcommand{\Br}{\mathrm{Br}}

\newcommand{\NS}{\mathrm{NS}}
\newcommand{\dv}{\mathrm{div}}

\newcommand{\OO}{\mathrm{O}}

\newcommand{\mon}{\mathrm{Mon}}

\newcommand{\W}{\mathsf{W}}
\newcommand{\kn}{K3^{[n]}}
\newcommand{\Lambdan}{{\Lambda}}
\newcommand{\psin}{{\psi^{[n]}}}
\newcommand{\rathodge}[1]{\psi_{#1}}
\newcommand{\cc}{c}
\newcommand{\nn}{{[n]}}
\newcommand{\cEnd}{\mathcal{E}nd}

\DeclarePairedDelimiter{\brak}{\langle}{\rangle}

\author[N.~Bignami]{Nicol\`o Bignami}
\address{SISSA, Via Bonomea 265, 34136 Trieste, Italy}
\email{\url{nbignami@sissa.it}}

\author[L.~Buelli]{Ludovica Buelli}
\address{Dipartimento di Matematica (DIMA), Via Dodecaneso 35, 16146 Genova, Italy}
\email{\url{ludovica.buelli@edu.unige.it}, \url{ludovica.buelli@hotmail.com}}

\author[I.~Mac\'ias Tarr\'io]{Irene Mac\'ias Tarr\'io}
\address{\parbox{0.9\textwidth}{Facultat de Matemàtiques i Informàtica \\[1pt]
Gran Via de les Corts Catalanes 585, 08005, Barcelona, Spain
\vspace{1mm}}}
\email{{\url{irene.macias@ub.edu}}}

\author[R.~Vacca]{Roberto Vacca}
\address{\parbox{0.9\textwidth}{Dipartimento di Matematica Università degli Studi di Roma Tor Vergata \\[1pt]
Via della Ricerca Scientifica, 1 - 00133, Roma
\vspace{1mm}}}
\email{{\url{robertovacca.math@gmail.com}}}

\author[V.~Zuliani]{Vanja Zuliani}
\address{\parbox{0.9\textwidth}{Universit\'e Paris-Saclay, CNRS, Laboratoire de Math\'ematiques d'Orsay\\[1pt]
Rue Michel Magat, B\^at. 307, 91405 Orsay, France
\vspace{1mm}}}
\email{\url{vanja.zuliani@universite-paris-saclay.fr}}

\title{On Moduli spaces of vector bundles on $K3^{[n]}$-type IHS manifolds}

\begin{document}
\begin{abstract}
We study moduli spaces of modular vector bundles on projective irreducible holomorphic symplectic manifolds of $K3^{[n]}$-type. Under suitable numerical assumptions, we exhibit connected components of these moduli spaces which are again irreducible holomorphic symplectic manifolds of $K3^{[n]}$-type. Moreover, the corresponding  universal families induce derived equivalences with the original manifolds.
This produces smooth components of moduli spaces of modular vector bundles on irreducible holomorphic symplectic manifolds of any even dimension.
\end{abstract}
\maketitle
{\hypersetup{linkcolor=black}\tableofcontents}
\section*{Introduction}

Moduli spaces of sheaves on $K3$ surfaces are one of the richest sources of examples of irreducible holomorphic symplectic (IHS, for short) manifolds, and the aim of this paper is to construct higher-dimensional analogues of this phenomenon. 
Let $Y$ be a projective IHS manifold of $K3^{[n]}$-type, and let $w\in \NS(Y)$ be a primitive class. Under explicit numerical assumptions on $w$, we construct modular vector bundles on $Y$ whose first Chern class is a prescribed multiple of $w$, and whose rank and discriminant are fixed by the construction. Then, for a suitable polarisation $H$, we prove that these sheaves are slope-stable and correspond to a whole connected component $X$ of the moduli space (see Theorem~\ref{thm_main_intro}). Each such component is again an IHS manifold of $K3^{[n]}$-type. Moreover, the corresponding  universal family induces a  derived equivalence between $X$ and $Y$ and a Hodge isometry $\H^2(X,\mathbb{Q})\xrightarrow{\sim} \H^2(Y,\mathbb{Q})$.

The geometry of moduli spaces of vector bundles on smooth projective varieties of dimension greater than two is largely inaccessible. In the case of $K3$ surfaces, the foundational work of Mukai \cite{Mukai1} shows that, provided that the Mukai vector $v$ is primitive and the polarisation is $v$-generic, moduli spaces of Gieseker-semistable sheaves are smooth and carry a holomorphic symplectic form. Many subsequent attempts to obtain general results for higher-dimensional IHS manifolds impose suitable topological conditions designed to mirror different aspects of the theory on $K3$ surfaces.

Verbitsky \cite{verbitsky-hyperholomorphic-bundles} introduced the notion of \emph{projectively hyperholomorphic} vector bundles. Roughly speaking, these objects are vector bundles whose projectivisation deforms along a twistor line as a fibration in projective spaces. This property automatically holds for simple vector bundles on $K3$ surfaces, by classical deformation theory (see \cite[Corollary 2.1.5]{ogrady-moduli-sheaves-hk}), but becomes more restrictive on higher-dimensional IHS manifolds. Verbitsky showed that it is satisfied by polystable vector bundles $E$, such that the induced Hermite--Einstein connection on $\mathcal{E}nd(E)$ is compatible with all the complex structures on the fixed twistor line. A remarkable feature of moduli spaces of projectively hyperholomorphic bundles is that, if they are smooth, they carry a holomorphic symplectic form, \cite[Section 9]{verbitsky-hyperholomorphic-bundles}.

On the other hand, O'Grady \cite{OG22} introduced the class of torsion-free modular sheaves. A torsion-free sheaf $F$ on a $2n$-dimensional IHS manifold $Y$ is \emph{modular} if there is a constant $d(F)\in \QQ$ such that, for all $\alpha\in \H^2(Y,\CC)$, it holds
    \begin{equation*}
        \int_Y \Delta(F)\cdot\alpha^{2n-2}=
        d(F)(2n-3)!!q_Y(\alpha)^{n-1},
    \end{equation*}
    where $q_Y(-)$ is the Beauville--Bogomolov--Fujiki form and the class 
    \[\Delta(F)=2\rk(F)\cc_2(F)-(\rk(F)-1)\cc_1(F)^2 \in \H^4(Y,\QQ)\]
    is the \textit{discriminant} of $F$. As observed by O'Grady, if $\Delta(F)$ is a multiple of $\cc_2(Y)$, then $F$ is modular. 
All the known examples of projectively hyperholomorphic sheaves satisfy this property.
Moreover, modular sheaves carry a wall-and-chamber decomposition of the ample cone, depending only on their Chern character, parallel to the $K3$ surfaces case (see \cite[Section 3]{OG22}).

The theory of projectively hyperholomorphic and modular sheaves has led to several spectacular applications. 
One of the first was Markman's proof of the Shafarevich conjecture for IHS manifolds of $K3^{[n]}$-type (see \cite{markman}). 
Subsequently, Markman's techniques were used to obtain new bounds for the period--index conjecture in \cite{hotchkissetal},
and to prove a version of the $\mathrm{D}$-equivalence conjecture in \cite{dequivconj}. Furthermore, in \cite{zhang} Zhang uses related techniques to study equivalences of twisted derived categories of IHS varieties of $K3^{[n]}$-type.

Recent literature in this area has developed along two main lines: the existence of modular vector bundles (both rigid and non-rigid, see \cite{OG22}, \cite{markmanstable}, \cite{Fat24}, \cite{FO24}, \cite{FT25}) on IHS manifolds of a fixed deformation type, and the subsequent construction of IHS manifolds as moduli spaces of such objects. 
Markman \cite{markmanstable} exhibited the first examples of non-rigid modular vector bundles on the Hilbert scheme of $n$ points on a K3 surface, induced by stable vector bundles on the underlying surface (Construction \ref{example1}).
On the other hand, a first substantial contribution to the second mentioned research line was provided by the work \cite{Bot24a, Bot24b} of Bottini, in which the OG10 IHS manifold is constructed as a moduli space of modular vector bundles on the Fano variety of lines on a smooth cubic fourfold. Furthermore, very recently, O'Grady \cite{OGrady2026,OG26c} studied moduli spaces  of modular vector bundles on manifolds of $K3^{[2]}$-type and constructed components which  are IHS manifolds of  $K3^{[(a^2+1)]}$-deformation type, for any integer $a\geq 1$.

The present work contributes to the second research line, aiming to construct IHS manifolds as moduli spaces of vector bundles on IHS manifolds. Let $Y$ be a projective IHS manifold of $K3^{[n]}$-type and $w\in \NS(Y)$ a primitive class. Let $\ell(w)\in(\mathbb{Z}/\dv(w)\mathbb{Z})^*/\{\pm1\}$ denote the class of the integer $m$ appearing in the decomposition $g(w)=v+m\delta_{S^{[n]}}$ after parallel transport $g\colon \H^2(Y,\ZZ)\to \H^2(S,\ZZ)\oplus \ZZ \delta_{S^{[n]}}$ to a Hilbert scheme $S^{[n]}$; this is well-defined by Definition~\ref{def_invariant}. Then our main Theorem reads as follows:
\begin{mainthm}[Theorem \ref{thm_main}] \label{thm_main_intro}Let $n\geq 2$ be an integer, let $Y$ be a projective IHS manifold of $K3^{[n]}$-type and let $w \in \NS(Y)$ be a primitive class. If $\dv(w)$ is odd, consider a positive even integer $r$ such that: \begin{enumerate}[label=(\alph*)]
    \item [{\normalfont{(a)}}]  $\left[\frac{r}{2}\right]=\ell(w)\in (\Z/\dv(w)\Z)^*/\{\pm 1\}$;
    \item [{\normalfont{(b)}}] $\frac{w^2}{2r}+ \frac{r}{4}(n-1)$ is a positive integer;
    \item [{\normalfont{(c)}}] $\frac{w^2}{2r}+ \frac{r}{4}(n-1)$ is coprime with $r$.
\end{enumerate}
If $\dv(w)$ is even, consider a positive integer $r\geq 3$ such that:
\begin{enumerate}[label=(\alph*')]
    \item [{\normalfont{(a')}}]  $[r]=\ell(w)\in (\Z/\dv(w)\Z)^*/\{\pm 1\}$;
    \item [{\normalfont{(b')}}] $\frac{w^2}{2r} + r(n-1)$ is a positive integer;
    \item [{\normalfont{(c')}}] $\frac{w^2}{2r}$ is coprime with $r$.
\end{enumerate}
Then, there exists a polarisation $H$ on $Y$ and a connected component $X$ of the moduli space $M_{Y,H}$ of Gieseker $H$-semistable sheaves on $Y$, such that 
\begin{enumerate}[label=$\bullet$]
    \item $X$ is an IHS manifold of $K3^{[n]}$-type.

    \item $X$ parametrises $H$-slope stable modular vector bundles $E$ on $Y$ with the following invariants:

    \begin{minipage}[l]{0.40\textwidth}
    \begin{equation*}
        \cc_1(E)= \bigg\{\begin{array}{ll}
           r^{n-1}n! w  & \textit{if }\dv(w) \textit{ is odd} \\
           r^{n-1}n! \frac{w}{2}  & \textit{if }\dv(w) \textit{ is even}
        \end{array}
    \end{equation*}
\end{minipage}
\hspace{1ex}
\begin{minipage}[l]{0.35\textwidth}
    \begin{align*}
       \rk(E)&= r^n n!\\
        \Delta(E) &= \frac{(r^n n!)^2}{12}\cc_2(Y).\end{align*}
\end{minipage}
\vspace{0.5ex}

    \item The  Fourier-Mukai transform $$\Phi_\kE \colon \D(X) \to \D(Y)$$ with kernel $\kE$ is an equivalence, where $\mathcal{E}$ is a  universal family on $X\times Y$.

\end{enumerate}
\end{mainthm}

\subsection*{Idea of proof}
We exhibit moduli spaces by constructing the corresponding  universal families. This is achieved by deforming a class of projectively hyperholomorphic vector bundles $\mathcal{U}^{[n]}$ on products of Hilbert schemes $M^{[n]}\times S^{[n]}$, where $M$ is a moduli space of isotropic slope-stable vector bundles on a K3 surface $S$, and hence is itself a K3 surface. We deform $\mathcal{U}^{[n]}$, as a projectively hyperholomorphic bundle, along diagonal twistor lines in the moduli space of marked pairs, as in \cite{markman}. 
Under suitable numerical conditions, for $Y$ of $K3^{[n]}$-type, we find $X$ of  $K3^{[n]}$-type such that $X\times Y$ carries a projectively hyperholomorphic vector bundle $\mathcal{E}$ . 
This step is achieved by combining lattice-theoretic arguments with Markman's monodromy description for manifolds of $K3^{[n]}$-type. The  derived equivalence $\D(X)\xrightarrow{\sim} \D(Y)$, as well as the Hodge isometry $\H^2(X,\mathbb{Q})\xrightarrow{\sim} \H^2(Y,\mathbb{Q})$ (see Corollary \ref{maincor3}), are part of Markman's package~\cite{markman}. Finally, exploiting the symmetries of the Hermite--Einstein connection on $\mathcal{E}nd(\mathcal{E})$, we show that there exists a polarisation such that $\mathcal{E}|_{\{x\}\times Y}$ is slope stable for all $x\in X$. It follows that $X$ is isomorphic to a component of the moduli space and $\mathcal{E}$ is a  universal family.

\subsection*{Corollaries and open questions}
The derived equivalence stated in Theorem \ref{thm_main_intro}
allows us to compare the Hodge structures of $Y$ and the moduli space $X$ appearing in the statement of the main result, using results in \cite{paolo_huy,markman}. The outcome is summarised in the two Corollaries below.

\begin{maincor}[Corollary \ref{cor_1}] In the setup of Theorem \ref{thm_main_intro}, 
the correspondence 
\begin{equation*}
		\tau_{\kE}\colon \H^{\ast}(X,\QQ)\to \H^{\ast}(Y,\QQ),
	\end{equation*} defined by the class  $\ch(\kE)\sqrt{\td_{X\times Y}}$, is an isomorphism of level zero Hodge structures (\ref{HSuntwisted}) and an isometry with respect to the Mukai pairing.
\end{maincor}

\begin{maincor}[Corollary \ref{cor_2}]\label{maincor3}
	In the setup of Theorem \ref{thm_main_intro}, there is a rational Hodge isometry
	\begin{equation*}
		 \rathodge{\kE}\colon\H^2(X,\mathbb{Q})\xrightarrow{\sim} \H^2(Y,\mathbb{Q})
	\end{equation*}
	which is a nontrivial rational multiple of the composition
	\begin{equation*}
		\H^2(X,\QQ)\xrightarrow{\cup \cc_2(X)^{n-1}}\H^{4n-2}(X,\QQ)\overset{\phi_{\kE}}{\longrightarrow}\H^2(Y,\QQ)
		\end{equation*}
		where $\phi_{\kE}$ is the correspondence induced by $\kappa(\kE)\sqrt{\td_{X\times Y}}$.    
\end{maincor}
Finally, in Section \ref{sec_examples} we provide applications of Theorem \ref{thm_main_intro}. 
We denote by ${}^{n}\mathcal{M}_{m}^{(\gamma)}$ the moduli space of polarised IHS varieties $(Y,h)$ of $K3^{[n]}$-type such that $\dv(h)=\gamma\in \ZZ_{\geq 1}$ and $h^2=2m\geq 2$. If $\gamma=1$ or $\gamma=2$ and $n+m\equiv 1\pmod{4}$, then ${}^{n}\mathcal{M}_{m}^{(\gamma)}$ is irreducible of dimension 20 and quasi-projective.
\begin{maincor}[Corollaries~\ref{cor_div1},~\ref{cor_div2}]\label{cor_intro_mainthm_alldim}
    Let $(Y,w)\in {}^{n}\mathcal{M}_{m}^{(1)}$ (resp. $(Y,w)\in {}^{n}\mathcal{M}_{m}^{(2)}$) and $r\in 2\ZZ_{\geq 1},k\in \ZZ_{\geq 1}$ such that $(k,r)=1$ and $m=r(2k-r/2(n-1))\geq 1$ (resp.  $r\in \ZZ_{\geq 3}$ such that $k:w^2/(2r)\in \ZZ$ and $(r,k)=1$).
    Then, there there is a polarisation $H$ on $Y$ and a connected component $X$ of the moduli space $M_{Y,H}$ of Gieseker $H$-semistable sheaves on $Y$, such that
    \begin{enumerate}
        \item [{\normalfont{(1)}}] $(X,\rathodge{\kE}(w))\in {}^{n}\mathcal{M}_{m}^{(1)}$ (resp. $(X, \rathodge{\kE}(w))\in {}^{n}\mathcal{M}_{m}^{(2)}$) for generic $Y$;
        \item [{\normalfont{(2)}}] $X$ parametrises $\mu_H$-stable vector bundles with of rank $r^n n!$, 
        first Chern class $r^{n-1}n!w$ \linebreak[4](resp. $r^{n-1}n!w/2$) 
        and discriminant $\frac{(r^n n!)^2}{12}\cc_2(Y)$;
        \item [{\normalfont{(3)}}] $X$ and $Y$ are derived equivalent.
    \end{enumerate}
\end{maincor}
In particular, we study some examples of locally complete families of $K3^{[2]}$-type. 

\begin{maincor}[{Beauville--Donagi, Debarre--Voisin, Iliev--Ranestad, see Corollary \ref{cor_BD_DV_IR}}] Let $(Y,w)$ be a  polarised IHS manifold  of $K3^{[2]}$-type such $\dv(w)=2$ and one of the following holds:
\begin{itemize}
    \item [{\normalfont{(a)}}] $w^2=6$,
    \item [{\normalfont{(b)}}] $w^2=22$, 
    \item [{\normalfont{(c)}}] $w^2=38$.
\end{itemize}
Then there is an ample divisor $H$ on $Y$ and a connected component $X$ of the moduli space of Gieseker $H$-semistable sheaves such that $X$ is of $K3^{[2]}$-type, derived equivalent to $Y$ and parametrises $H$-slope stable vector bundles $E$ with
\begin{itemize}
    \item [{\normalfont{(a')}}] $rk(E)=18,\quad c_1(E)=3h \quad \Delta(E)=27\cc_2(Y)$ in case $(a)$,
    \item [{\normalfont{(b')}}] $rk(E)=2\cdot11^2,\quad c_1(E)=11h\quad \Delta(E)=\frac{11^4}{3}\cc_2(Y)$ in case $(b)$,
    \item [{\normalfont{(c')}}] $rk(E)=2\cdot19^2,\quad c_1(E)=19h,\quad \Delta(E)=\frac{19^4}{3}\cc_2(Y)$ in case $(c)$.
\end{itemize}
\end{maincor}
\noindent Analogous conclusions are drawn in Corollary~\ref{cor_EPWsextic} for the case of the EPW sextic.

In Corollary \ref{cor_intro_mainthm_alldim} we observe that the Chern characters of the of the vector bundles in the moduli space $X$ do not depend on the choice of $(Y,w)\in {}^{n}\mathcal{M}_m^{(\gamma)}$, hence it is natural to ask about properties of the construction in families.
Let us consider the quasi-projective coarse moduli space ${}^{n}\mathcal{M}_m^{(\gamma)}$ of polarised IHS manifolds $(Y,w)$ of $K3^{[n]}$-type with $\gamma\in \{1,2\}$ and $n+m\equiv 1 \pmod{4}$ if $\gamma=2$.
\begin{mainquest}
    Does the correspondence $Y\mapsto X(Y)$ define a birational automorphism between moduli spaces 
    \[
    \mu\colon {}^{n}\mathcal{M}_m^{(\gamma)} \dashrightarrow { {}^{n}\mathcal{M}_{m}^{(\gamma)}}\;?
    \]
\end{mainquest}
\noindent The answer to the question is affirmative in the case of the EPW sextics (see Corollary~\ref{cor_EPWsextic},  \cite[Section 4.2]{OG15} and  \cite[Proposition 5.2]{kapustka-kapustka-derived-equivalent-hk4}).

\quad

It is natural to ask to what extent the arguments of this paper can be generalized. Although we do not pursue this direction here, we expect that the same techniques can also produce moduli spaces of twisted vector bundles (see Remark~\ref{twistedmoduli}). In a different direction, one possible strategy would be to start from a broader class of projectively hyperholomorphic vector bundles, generalising Construction~\ref{example1}.

Let $Y_0$ be an IHS manifold endowed with a K\"ahler class $\omega_{Y_0}$, and let $M$ be a fine moduli space of $\omega_{Y_0}$-slope stable projectively hyperholomorphic vector bundles on $Y_0$. Assume that $M$ is smooth. Then $M$ carries a natural holomorphic symplectic form (see \cite{verbitsky-hyperholomorphic-bundles} and \cite[Section~7]{Bot24b}). Let $\mathcal E$ be the universal family on $M\times Y_0$. Suppose that the cohomological correspondence $\psi_{\kE}$ is a Hodge isometry such that $\omega_M:=\psi_{\kE}^{-1}(\omega_{Y_0})$ is a K\"ahler class on $M$. The K\"ahler classes $\omega_M$ and $\omega_{Y_0}$ then determine a diagonal hyperk\"ahler structure on $M\times Y_0$. This leads to the following question.

\begin{mainquest}\label{quest2}
Is the vector bundle $\mathcal E$ projectively hyperholomorphic with respect to the split K\"ahler class $\pr_M^*\psi_{\mathcal E}^{-1}(\omega_{Y_0})+\pr_{Y_0}^*\omega_{Y_0}$ on $M\times Y_0$?
\end{mainquest}
An affirmative answer would allow one to deform $\mathcal E$ as a projectively hyperholomorphic bundle along the corresponding diagonal twistor line, and would therefore produce moduli spaces of possibly twisted vector bundles on deformations of $Y_0$.
It would be very interesting to understand if the answer to Question~\ref{quest2} is positive for the examples studied in \cite{krug-reede-zhang}.

\subsection*{Structure of the paper}
In Section~\ref{secprojhyperholo} we review the theory of projectively hyperholomorphic vector bundles and provide a criterion for the Fourier--Mukai kernel of a twisted derived equivalence to be a twisted universal family for a moduli space.
In Section~\ref{sec_admissible} we take advantage of the explicit description of monodromy operators for IHS manifolds of $K3^{[n]}$-type and prove some lattice theoretical Lemmata. These will translate into conditions on the $K3^\nn-$type IHS manifolds to which Markman's and Verbitsky's results can be applied. 
In Section~\ref{sec_mainthm} we state and prove Theorem~\ref{thm_main_intro}.
We conclude with Section~\ref{sec_examples}, where we present examples to which Theorem~\ref{thm_main_intro} applies.
In the appendices, we collect some auxiliary results on hyperholomorphic (possibly twisted) vector bundles and on the variation of slope stability in the K{\"a}hler cone, which are of independent interest.

\subsection*{Acknowledgments}
This project started during the summer school PRAGMATIC 2025 in
Catania; we thank the organisers Elena Guardo, Alfio Ragusa, Francesco Russo, Giovanni
Staglianò and Giuseppe Zappalà for their kind hospitality, the stimulating research environment
and the financial support during the school. We are especially grateful to Alessio Bottini for suggesting the problem and for his encouragement. We would like to thank Pietro Beri, Alessio Bottini, Emanuele Macrì, Alina Marian and Andrea Ricolfi for helpful discussions. L.B. would like to thank Arvid Perego for the useful opinions and precious advices. 
R.V. warmly thanks SISSA and ICTP for their hospitality.
V.Z. was supported by the ERC Synergy Grant 854361 HyperK.

\subsection*{AI disclosure}
We used free versions of LLMs to check for grammar mistakes in the text. All the contents of this article were conceived and written by the authors.

\subsection*{Notation}
Let $X$ be a complex manifold, and let $\mathcal{F}$ be a sheaf on $X$ in the analytic topology. We denote by $\H^*_{\mathrm{an}}(X,\mathcal{F})$ the corresponding sheaf cohomology. If $R$ is a ring, we denote by $\H^*(X,R)$ the singular cohomology of $X$ with coefficients in $R$. If $X$ is K\"ahler, we denote by $\mathcal{K}_X\subset \H^{1,1}(X)\cap \H^2(X,\RR)$ the K\"ahler cone.

\section{Projectively hyperholomorphic bundles}\label{secprojhyperholo}
The proof of our main result is based on the work of Markman on projectively hyperholomorphic bundles (see \cite{markman}). In this Section, we recall the main features that we need in our arguments. For other expositions and applications of this theory, we refer to \cite{dequivconj, ogrady-moduli-sheaves-hk, kapustka-kapustka-derived-equivalent-hk4}.

\subsection{Verbitsky's deformability results}
Let $X$ be an IHS manifold and let $E$ be a holomorphic vector bundle on $X$. To every K\"ahler class $\omega \in \mathcal{K}_X$, one can associate the twistor family $\mathcal{X} \rightarrow \mathbb{P}^1_\omega$, a smooth and proper morphism of complex manifolds whose fibers are the complex manifolds obtained by rotating the complex structure of $X$ within the hyperk\"ahler structure determined by $\omega$ (see \cite[Section~1.13]{Huy-HK-basic}). In order for $E$ to deform as a holomorphic vector bundle over the whole twistor family, its Chern classes must remain of Hodge type for all complex structures parametrised by $\mathbb{P}^1_\omega$. Verbitsky proved that, provided that $E$ is $\omega$-slope stable, this is the only obstruction. More precisely, using the notation above, he gives the following definition.

\begin{df}
A holomorphic vector bundle $E$ on an IHS manifold $(X,\omega)$ equipped with a Kähler class $\omega$ is called \textit{$\omega$-stable hyperholomorphic} if it is $\omega$-slope stable and its Chern classes $\cc_1(E)$ and $\cc_2(E)$ remain of Hodge type for all complex structures in the twistor line $\mathbb{P}^1_\omega$.
\end{df}

The following Theorem of Verbitsky is a strong deformability result for such bundles.

\begin{thm}[{\cite[Theorem~3.19]{verbitsky-hyperholomorphic-sheaves}}]\label{verb}
If $E$ is $\omega$-stable hyperholomorphic, then there exists a holomorphic vector bundle $\widetilde{E}$ on the twistor family $\mathcal{X}\rightarrow\mathbb{P}^1_\omega$ such that $\widetilde{E}|_{\mathcal{X}_0}\simeq E$, where $\mathcal{X}_0$ denotes the fiber corresponding to the original complex structure on $X$.
\end{thm}

In other words, the Hodge-theoretic condition on $\cc_1(E)$ and $\cc_2(E)$ guarantees that the Hermite--Einstein connection on $E$ is compatible with all complex structures in the twistor line. This produces a holomorphic vector bundle on the total space of the twistor family whose restriction to the original fiber is the given bundle $E$.

Verbitsky's result was generalised to Azumaya algebras in \cite{markman_bbf}. We now recall the definitions needed to state Markman's generalization.
\begin{df}
An \emph{Azumaya algebra} $\kA$ of rank $r^2$ over a complex manifold $X$ is a locally free coherent $\mathcal{O}_X$-module, with a global section $1_\kA$ and an $\mathcal{O}_X$-linear associative multiplication $m \colon \kA \otimes_{\mathcal{O}_X} \kA \longrightarrow \kA$ with identity $1_\kA$, such that there exists an open covering $\{U_i\}$ of $X$ and a collection of isomorphisms $\eta_i \colon \kA|_{U_i} \longrightarrow \mathcal{E}nd(F_i)$ of unital associative algebras for some locally free $\mathcal{O}_{U_i}$-module $F_i$ of rank $r$ over each $U_i$.
\end{df}
\begin{rmk}
For other models of Azumaya algebras, we refer to \cite[Section 2]{van-bree-gholampour-jiang-kool-pgl-sl}. In particular, when $X$ is a smooth projective variety, they correspond to Severi--Brauer varieties.
\end{rmk}

Analogously to the case of vector bundles, one can define $\omega$-slope stability for Azumaya algebras (see \cite[Definition~6.5]{markman_bbf}). The generalisation of Verbitsky's result then reads as follows:
\begin{thm}[{\cite[Corollary~6.12]{markman_bbf}}]\label{deformingazumaya}
Let $\mathcal{A}$ be an $\omega$-slope stable Azumaya algebra on an IHS manifold $X$. Assume moreover that $\cc_2(\mathcal{A})$ remains of Hodge type $(2,2)$ along the $\omega$-twistor line, and that the first Chern class of every direct summand of $\mathcal{A}$ vanishes. Then there exists an Azumaya algebra $\widetilde{\mathcal{A}}$ on the twistor family $\mathcal{X} \rightarrow \mathbb{P}^1_\omega$ such that $\widetilde{\mathcal{A}}|_{\mathcal{X}_0} \simeq \mathcal{A}$, where $\mathcal{X}_0$ denotes the fiber corresponding to the original complex structure on $X$.\end{thm}

\begin{rmk}
For simplicity, we stated Verbitsky's results in the context of irreducible holomorphic symplectic manifolds, although his arguments only require a hyperk\"ahler structure. For the more general setting, we refer to Appendix~\ref{app_hyperholo}. It is crucial to use this more general definition, as we will work with projectively hyperholomorphic bundles on a product of IHS manifolds, which is not itself irreducible holomorphic symplectic, but can be endowed with a diagonal hyperk\"ahler structure (see \cite[Section 5.3]{markman}).
\end{rmk}

\subsection{Invariants of Azumaya algebras} An Azumaya algebra $\mathcal{A}$ of rank $r^2$ on a complex manifold $X$ can be regarded as a $\PGL_r$-torsor. Via the boundary map $\H^1_{\mathrm{an}}(X,\PGL_r)\rightarrow \H^2_{\mathrm{an}}(X,\mathcal{O}^\times_X)$, it defines an element $\alpha\in \H^2_{\mathrm{an}}(X,\mathcal{O}^\times_X)$, called \emph{Brauer class}. It is always torsion and it has order dividing $r$. The group $\Br'(X):=\H^2_{\mathrm{an}}(X,\mathcal{O}^\times_X)_\mathrm{tors}$ is called \emph{cohomological Brauer group}.

\begin{rmk}
The \emph{Brauer group} $\Br(X)$ is the group of Brauer equivalence classes of Azumaya algebras on $X$ (see \cite{schroer-analytic-brauer}). The construction above induces a natural group isomorphism $\Br(X)\longrightarrow \Br'(X)$. This map is known to be an isomorphism in many cases; for instance, in the algebraic setting, it is an isomorphism for smooth projective varieties by Gabber's Theorem (see \cite{dejong-gabber}). In general, however, the Brauer group and the cohomological Brauer group need not be automatically isomorphic. This distinction will not play any role here, since all twists considered below arise from Azumaya algebras, and therefore lie in the image of the natural map $\Br(X)\longrightarrow \Br'(X)$.
\end{rmk}
Let $\alpha\in\Br'(X)=\H^2_{\mathrm{an}}(X,\mathcal{O}_X^\times)_{\mathrm{tors}}$ be represented by a \v{C}ech $2$-cocycle $\{\alpha_{ijk}\}$ with respect to an open cover $\{U_i\}$ of $X$.

\begin{df}
An \emph{$\alpha$-twisted locally free sheaf} is given by a collection of locally free sheaves $F_i$ on $U_i$ and isomorphisms $\varphi_{ij}\colon F_j|_{U_{ij}}\longrightarrow F_i|_{U_{ij}}$ such that, on triple overlaps $U_{ijk}$, we have $\varphi_{ij}\circ\varphi_{jk}\circ\varphi_{ki}=\alpha_{ijk}\cdot\operatorname{id}_{F_i}$.
\end{df}

There is a correspondence between Azumaya algebras and twisted locally free sheaves (or equivalently $\PGL_r$-bundles) (see \cite{huybrechts-schroeer-brauer}). Given an $\alpha$-twisted vector bundle $E$, it is easy to see that $\mathcal{E}nd(E)$ is an Azumaya algebra with Brauer class $\alpha$. Conversely, given an Azumaya algebra $\mathcal{A}$, by definition there exists an open cover $\{U_i\}$ of $X$ such that $\mathcal{A}|_{U_i} \simeq \mathcal{E}nd(E_i)$ for some vector bundles $E_i$ on $U_i$. There are isomorphisms of algebras $\phi_{ij}:\mathcal{E}nd(E_i)\rightarrow\mathcal{E}nd(E_j)$ that - up to refining the covering - can be assumed to be induced by isomorphisms $\overline{\phi}_{ij}\colon E_i\rightarrow E_j$, by the Skolem--Noether Theorem. The twisted sheaf associated to $\mathcal{A}$ is uniquely determined up to tensoring with an untwisted holomorphic line bundle.

We are interested in the case where $X$ is a compact K\"ahler manifold whose odd integral cohomology groups vanish (e.g. an IHS manifold of $K3^{[n]}$-type or a product of IHS manifolds of $K3^{[n]}$-type). In this case, via the exponential sequence, there is an identification $$\Br'(X)\simeq (\H^2(X,\mathbb{Z})/\Pic(X))\otimes_{\mathbb{Z}}\mathbb{Q}/\mathbb{Z}.$$ In other words, any Brauer class can be represented by a rational cohomology class $B\in \H^2(X,\mathbb{Q})$, which we call a \emph{$B$-field}.

\begin{rmk}
Our main construction consists in deforming Azumaya algebras as in Theorem~\ref{deformingazumaya}. Let $X_1, X_2$ be two deformation equivalent projective IHS manifolds, and suppose that $E_1\in\Coh(X_1)$ is an untwisted locally free sheaf. Under deformation of the associated Azumaya algebra $\cEnd(E_1)$, the Chern classes of $E_1$ need not be preserved: if $E_2\in\Coh(X_2)$ is an untwisted locally free sheaf obtained by deforming $\cEnd(E_1)$ along a chain of twistor paths from $X_1$ to $X_2$, then the Chern classes of $E_2$ need not be obtained from those of $E_1$ by parallel transport, even after tensoring by a line bundle on $X_2$. A simple example is provided by a line bundle $L$ on an IHS manifold $X_1$. Then $\cEnd(L)\simeq \mathcal{O}_X$, which deforms trivially along any twistor family. Along a twistor family, one can reach a projective  manifold $X_2$ such that $\cc_1(L)$ is not algebraic. On the other hand, it holds $\kO_{X_2}=\cEnd(T)$ for any line bundle $T\in \Pic(X_2)$. Hence $T$ is a deformation of $L$ as projectively hyperholomorphic sheaves, but $\cc_1(T)-\cc_1(L)$ is not algebraic on $X_2$.
\end{rmk}

Motivated by the previous remark, we recall invariants that are better behaved. Let $E$ be a twisted vector bundle of rank $r$ on a complex manifold $X$. Since $E^{\otimes r}\otimes \det(E)^{-1}$ is an untwisted vector bundle, we can define the class \begin{equation}\label{eq_def_kappa}
    \kappa(E):=\sqrt[r]{\ch(E^{\otimes r}\otimes \det(E)^{-1})}\in \H^*(X,\mathbb{Q}),
\end{equation} 
where the $r$-th root is chosen to have degree-zero component equal to $r$. 
Note that the Azumaya algebra $\cEnd(E)$ determines $E$ up to twist by a holomorphic line bundle $L$ and $\kappa(E\otimes L)=\kappa(E)$. Consequently, the class $\kappa(E)$ depends only on the Azumaya algebra $\mathcal{E}nd(E)$, and can be regarded as an invariant of $\mathcal{E}nd(E)$.

\begin{rmk}\label{stiefel-whitney}
    Let $\mathcal{A}$ be an Azumaya algebra of rank $r^2$ on a complex manifold $X$. It determines a $\PGL_r$-torsor, hence $\mathcal{A}$ defines an element of $\H_{\mathrm{an}}^1(X,\PGL_r)$. The central extension of sheaves of groups, in the analytic topology,
    \[
        1 \rightarrow \mu_r \rightarrow \SL_r \rightarrow \PGL_r \rightarrow 1,
    \]
    induces the boundary map $\delta\colon \H_{\mathrm{an}}^1(X,\PGL_r)\rightarrow \H_{\mathrm{an}}^2(X,\mu_r)$. As a consequence, we can associate to $\mathcal{A}$, the so-called \textit{Stiefel--Whitney class} $w(\mathcal{A}):=\delta(\mathcal{A})$. Since $X$ is locally contractible and paracompact, there is a natural identification  $\H_{\mathrm{an}}^2(X,\mu_r)\simeq \H^2(X,\mathbb{Z}/r\mathbb{Z})$. We point out that the Stiefel--Whitney class appears in the work of C\u{a}ld\u{a}raru under the name of \textit{topological twisting class} (see \cite[Section 4]{caldararu-nonfine-k3}).
    We recall below some fundamental properties of this class.
    \begin{enumerate}
        \item If $\mathcal{A}$ is an Azumaya algebra on a product of complex manifolds $X\times Y$, then for all $x\in X$, it holds $w(\mathcal{A}|_{\{x\}\times Y})=i_x^{*}w(\mathcal{A})$, where $i_x\colon \{x\}\times Y=Y\hookrightarrow X\times Y$ is the inclusion and $i_x^*\colon \H^2(X\times Y,\mathbb{Z}/r\mathbb{Z})\rightarrow \H^2(Y,\mathbb{Z}/r\mathbb{Z})$ denotes the pull-back.
        \item The Stiefel--Whitney class remains constant in a deformation family of Azumaya algebras. More precisely, if $\mathcal{X}\rightarrow B$ is a smooth and proper morphism of complex manifolds, with $B$ simply connected, and $\mathcal{A}$ is an Azumaya algebra of rank $r^2$ on $\mathcal{X}$, then for all $b,b'\in B$, we have $w(\mathcal{A}_b)=w(\mathcal{A}_{b'})$ under the natural identification $\H^2(\mathcal{X}_b,\mathbb{Z}/r\mathbb{Z})\simeq \H^2(\mathcal{X}_{b'},\mathbb{Z}/r\mathbb{Z})$ induced by parallel transport.

        \item  If $\mathcal{E}$ is an untwisted locally free sheaf such that $\mathcal{E}nd(\mathcal{E})=\mathcal{A}$, then $w(\mathcal{A})=-\pi(\cc_1(\mathcal{E}))$, where $\pi\colon \H^2(X,\mathbb{Z})\rightarrow \H^2(X,\mathbb{Z}/r\mathbb{Z})$.
    \end{enumerate} 
\end{rmk}

\begin{rmk}\label{rmk_trasp_c1}
    Let $X_1,X_2$ be two deformation equivalent projective IHS manifolds. Suppose that an untwisted locally free sheaf $E_2\in\mathrm{Coh}(X_2)$ of rank $r$ is obtained from an untwisted locally free sheaf $E_1\in\mathrm{Coh}(X_1)$ by deforming the corresponding Azumaya algebras along a chain of twistor paths. By $(2)$ and $(3)$ of the previous remark, we have $\pi(\cc_1(E_2))=\pi(\cc_1(E_1))$ under the natural identification $\H^2(X_2,\mathbb{Z}/r\mathbb{Z})\simeq\H^2(X_1,\mathbb{Z}/r\mathbb{Z})$. If the parallel transport of $\cc_1(E_1)$ is algebraic on $X_2$, this is equivalent to $\cc_1(E_2)-\cc_1(E_1)\in r \Pic(X_2)$. Thus, after replacing $E_2$ by $E_2\otimes L$ for a suitable $L\in\Pic(X_2)$,  we may assume that $\cc_1(E_2)=\cc_1(E_1)$, where $\cc_1(E_1)$ is viewed on $X_2$ via parallel transport.
\end{rmk}

We are ready to give the following definition. Let $(X,\omega)$ be an IHS manifold with a fixed K\"ahler class.

\begin{df}
A $\omega$-stable (possibly twisted) holomorphic vector bundle is called \emph{projectively hyperholomorphic} if $\mathcal{E}nd(E)$ is $\omega$-polystable hyperholomorphic.
\end{df}

\subsection{Markman's projectively hyperholomorphic vector bundles} \label{secprojholo} The starting point of \cite{markman} is the construction of a remarkable class of projectively hyperholomorphic vector bundles.

\begin{construction}[{{\cite[Section 4]{ogrady-moduli-sheaves-hk}, \cite[Section 11]{markmanstable}, \cite[Section 7]{markman}}}]\label{example1}
    Let $S$ be a $K3$ surface with $\Pic(S)=\ZZ h$ and consider an isotropic Mukai vector $u=(r, mh, s)$ such that the moduli space $M:=M_S(u)$ is fine and parametrises slope-stable bundles. 
    For instance, if we assume $\gcd(r,m)=1=\gcd(r,s)$, then by \cite[Corollary 4.6.7, Remark 6.1.9]{HL10}, the moduli space $M$ satisfies the above-mentioned properties.  By \cite{Mukai2}, the moduli space $M$ is a $K3$ surface and we have a universal vector bundle $\mathcal{U}$ on $M\times S$. 
    By conjugating Bridgeland--King--Reid equivalence \cite{bridgeland-king-reid-mckay}, we get the existence of a vector bundle $\mathcal{U}^{[n]}$ on the product of the Hilbert schemes $M^{[n]}\times S^{[n]}$.
    More precisely, the vector bundle $\mathcal{U}^{[n]}$ is the Fourier--Mukai kernel of the derived equivalence
    \[
\Phi_{\U^{[n]}}:=BKR\circ\Phi_{\U^{\boxtimes n}}\circ BKR^{-1}\colon \D(M^{[n]})\to \D(S^{[n]}).
    \]
    Given $[Z]\in M^{[n]}$, corresponding to a reduced subscheme, the fiber ${\mathcal{U}^{[n]}}|_{\{[Z]\}\times S^{[n]}}$ is well understood (see \cite[Section 4.3]{ogrady-moduli-sheaves-hk}). Recall the construction of the isospectral Hilbert scheme, which is defined as the reduced fiber product associated to the diagram
    
\[\begin{tikzcd}
	{X_n(S)} & {S^n} \\
	{S^{[n]}} & {S^{(n)}}
	\arrow["\tau"', from=1-1, to=1-2]
	\arrow["\rho", from=1-1, to=2-1]
	\arrow["\pi"', from=1-2, to=2-2]
	\arrow["\gamma", from=2-1, to=2-2]
\end{tikzcd}\]
    where $\gamma$ is the Hilbert-Chow morphism and $\pi$ is the quotient map under the action of the $n-$th symmetric group $\Sigma_n$. Then, if $Z=\{p_1,\dots,p_n\}$ is a reduced $0$-dimensional subscheme of $M$, we have
    \[
    {\mathcal{U}^{[n]}}|_{\{[Z]\}\times S^{[n]}}\simeq \left(\rho_*\tau^*\bigoplus_{\sigma\in\Sigma_n}E_{p_{\sigma(1)}}\boxtimes\cdots\boxtimes E_{p_{\sigma(n)}}\right)^{\Sigma_n},
    \]
    where $E_{p_i}$ denotes the stable vector bundle corresponding to $p_i\in M$ and the exponent $\Sigma_n$ means that we are taking the $\Sigma_n$ invariant subsheaf with respect to the natural $\Sigma_n$ action (see \cite[Example 3.2.10]{ogrady-moduli-sheaves-hk}).
    In addition, for all $[Z]\in M^{[n]}$, we have the following identities
\begin{align}
\rk\!\left({\mathcal{U}^{[n]}}|_{\{[Z]\}\times S^{[n]}}\right)
    &= r^n n!, \label{rank1}\\
\cc_1\!\left({\mathcal{U}^{[n]}}|_{\{[Z]\}\times S^{[n]}}\right)
    &= r^{n-1}n! \, m h
       - r^n n! \, \frac{\delta_{S^{[n]}}}{2}, \label{c11}\\
\Delta\!\left({\mathcal{U}^{[n]}}|_{\{[Z]\}\times S^{[n]}}\right)
    &= \frac{(r^n n!)^2}{12}\cc_2(S^{[n]}), \label{eq_modularity1}
\end{align}
    under the decomposition  $\H^2(S^{[n]},\ZZ)=\H^2(S,\ZZ)\oplus \ZZ\delta_{S^{[n]}}$ (see \cite[Lemma 11.1]{markmanstable}, \cite[Example 3.2.10]{ogrady-moduli-sheaves-hk}, \cite[Remark 2.10]{OGrady2026}).
    In particular, the vector bundles ${\mathcal{U}^{[n]}}|_{\{[Z]\}\times S^{[n]}}$ are modular.
    \end{construction}
    
Next, we briefly describe the argument used by Markman in \cite[Theorem 8.4]{markman} to construct (possibly twisted) vector bundles over any IHS manifold of $K3^{[n]}$-type.
The key will be Verbitsky's deformation of hyperholomorphic bundles on twistor lines given in Theorem~\ref{verb}, but in the revisited form of Theorem~\ref{deformingazumaya} (see also \cite[Proposition 5.22]{markman} and \cite[Proposition 5.15]{markman}).

We start by considering the following moduli spaces (see \cite{buskin-hodge-isometry-k3} and \cite{markman}). Let 
$\mathfrak{M}_\Lambdan$ be the coarse moduli space parametrising marked IHS manifolds $(X,\eta_X)$, where $\eta_X\colon \H^2(X,\Z)\to \Lambdan$ is an isometry.
Moreover, for any fixed rational isometry $\phi:\Lambdan \otimes \QQ\to \Lambdan \otimes \QQ$, one can define the moduli space  $\mathfrak{M}_\phi$ of \textit{rational Hodge isometries}, parametrising quadruples $(X,\eta_X,Y,\eta_Y)$ such that $(X,\eta_X),(Y,\eta_Y)\in \mathfrak{M}_\Lambdan$ and the composition  Hodge isometry $\psi: =\eta_Y^{-1}\circ \phi\circ \eta_X \colon \H^2(X,\QQ)\to \H^2(Y,\QQ)$ is a Hodge isometry sending some open subcone of the K\"ahler cone of $X$ into the K\"ahler cone of $Y$.
This last condition allows us to consider \textit{diagonal twistor lines}.
Pick $(X,\eta_X,Y,\eta_Y)\in \mathfrak{M}_\phi$, a class $\omega\in \mathcal{K}_X$ such that $\psi(\omega)\in \mathcal{K}_Y$ and call $\mathcal{X} \rightarrow \mathbb{P}^1_\omega$ and $\mathcal{Y} \rightarrow \mathbb{P}^1_{\psi(\omega)}$ the associated twistor families.
As explained in \cite[§5.3]{markman}, we get an identification of $\PP^1_{\omega}$ with $\PP^1_{\psi(\omega)}$ so that $\mathcal{X}\times_{\PP^1}\mathcal{Y}$ is the twistor family associated to $X\times Y$ with K\"{a}hler class $\pi_X^*\omega+\pi_Y^*\psi(\omega)$.
This definition is fundamental in order to apply Verbitsky's result to families of sheaves. A crucial property is:
\begin{lem}[{\cite[Lemma~5.14]{markman}}]\label{lem_pisurj}
Fix a connected component $\Mf_{\phi}^{\circ}$ of $\Mf_{\phi}$.
For $i=1,2$ the natural projection morphism $\pi_i\colon \Mf_{\phi}^{\circ}\to \mathfrak{M}_{\Lambdan}^{\circ}$, defined by $\pi_i(X_1,\eta_1,X_2,\eta_2)= (X_i,\eta_i)$, is surjective onto a connected component $\mathfrak{M}_{\Lambdan}^{\circ}$ of $\mathfrak{M}_{\Lambdan}$. 
\end{lem}

By \cite[Lemma 5.8]{markman}, we can connect any two points in a fixed connected component of $\mathfrak{M}_\phi$ by a diagonal twistor path, i.e.~a connected chain of diagonal twistor lines (see \cite[Definition 5.6]{markman}).
Moreover, we can assume this path to be \textit{generic}, that is, we suppose the join of any two twistor lines to be points parametrising manifolds with trivial Picard group.
Choosing the appropriate  $\phi=\psin$ (see Section \ref{sec_vanishing_twist} for further details) and the markings, we can consider the connected component $\mathfrak{M}_\psin^\circ$ of $\mathfrak{M}_\psin$ that contains $P:=(M^{[n]},\eta_{M^{[n]}},S^{[n]},\eta_{S^{[n]}})$.
Now, given any $Q:=(X,\eta_X,Y,\eta_Y)\in\mathfrak{M}^\circ_\psin$, we can find a generic diagonal twistor path connecting $P$ to $Q$.
Then, the vector bundles $\cEnd(\mathcal{U}^{[n]})$ flatly deform as Azumaya algebras on those twistor paths and from them we can construct (possibly twisted) vector bundles on $X\times Y$, unique up to twisting by a line bundle.

More precisely, Markman considers a diagonal twistor line connecting $P$ to a point \linebreak[4]$P':=(M'^{[n]},\eta_{M'^{[n]}},S'^{[n]},\eta_{S'^{[n]}})$ consisting of a pair of marked Douady spaces with $\Pic(M')=0=\Pic(S')$, along which $\mathcal{U}^{[n]}$ deforms as an $\overline{\omega}$-slope stable bundle, for some K\"{a}hler class $\overline{\omega}$, then followed by a generic diagonal twistor path from $P'$ to $Q$.
Call $\mathcal{U}'^{[n]}$ the deformation of $\mathcal{U}^{[n]}$ on $P'$, which is $\omega$-projectively hyperholomorphic for any K\"{a}hler class $\omega$ (see \cite[Lemma 8.3]{markman} and Definition \ref{def_proj_hyperhol}).
Therefore, the vector bundle $\mathcal{U}'^{[n]}$ can be deformed as an hyperholomorphic bundle along any twistor line starting from $P'$.
We apply this to the first twistor line appearing in the path connecting $P'$ to $Q$.
Crucially, being this path generic, on each join of two twistor lines we obtain a (possibly twisted) vector bundle, which is slope-stable for some K\"{a}hler class, on a variety with trivial Picard group.
It follows that such a bundle has only trivial sub-objects; hence, it is slope-stable for all K\"{a}hler classes.
Therefore, the above argument can be iterated until we reach $Q$.
In conclusion, we get the following.

\begin{thm}[{\cite{markman},\cite[Theorem 2.3]{dequivconj}}]\label{thm_compatiblevectorbundle}
    Given $(X,\eta_X,Y,\eta_Y)\in\mathfrak{M}^\circ_{\psin}$ as above, there exists a split K\"ahler class $\omega_{X\times Y}:=\pi_X^*\omega_X+\pi_Y^*\omega_Y$ and a projectively $\omega_{X\times Y}$-stable hyperholomorphic vector bundle $(\kE,\alpha_\kE)$  such that
    \begin{enumerate}
        \item [{\normalfont{(1)}}] $\alpha_\mathcal{E}=\left[-\frac{\cc_1(\mathcal{U}^{[n]})}{\rk(\mathcal{U}^{[n]})}\right]$, where we view $\H^2(X\times Y,\mathbb{Z})$ as a trivial local system over the moduli of marked pairs;
        \item [{\normalfont{(2)}}] The twisted Fourier-Mukai transform
        \[
        \Phi_{\mathcal{E}}\colon \D(X,\alpha_X^{-1})\rightarrow \D(Y,\alpha_Y),
        \]
        with kernel $\kE$ is an equivalence, where $\alpha_\kE=\alpha_X+\alpha_Y\in \Br(X\times Y)=\Br(X)\oplus\Br(Y)$;
        \item [{\normalfont{(3)}}] For all $x\in X$, the Stiefel-Whitney class of the bundle on $Y$ corresponding to $x$ is
        \[w(\kE|_{\{x\}\times Y})=\left[\eta_Y^{-1}\left(r^{n-1}n!m\eta_S(h)-r^nn!\frac{\delta}{2}\right)\right]\in \H^2(Y,\mathbb{Z}/r^nn!\mathbb{Z}).\]
    \end{enumerate}
\end{thm}

The third statement follows easily from the properties in Remark~\ref{stiefel-whitney}.
Indeed, the vector bundle $\mathcal{E}$ is constructed deforming the Azumaya algebra $\mathcal{E}nd(\mathcal{U}^{[n]})$, for which the Stiefel-Whitney class is easily computed using \eqref{c11}.

\subsection{Components of the moduli space from derived equivalences} 
The next standard Lemma gives a criterion for identifying the base of a family of sheaves with a connected component of a moduli space.

\begin{lem}\label{lem_Ymodulisp}
    Let $X$ be a smooth projective variety and let $(Y,H)$ be a smooth polarised variety. Let $\Phi_{\kE}\colon \D(X,\alpha_{X}^{-1})\to \D(Y)$ be an equivalence, where $\kE$ is a $\pr_X^*\alpha_{X}$-twisted vector bundle on $X\times Y$ such that for every $x\in X$, the restriction $\kE_{x}$ is Gieseker stable with respect to $H$.
    Then $X$ is isomorphic to a connected component $M_{Y,H}(\ch(\kE_{x}))^{\circ}$ of the moduli space of Gieseker $H$-semistable sheaves on $Y$ with Chern character $\ch(\kE_{x})$, for any $x\in X$.
\end{lem}
\begin{proof}
We show that $X$ is isomorphic to a connected component of the moduli space $M:=M_{Y,H}\!\bigl(\ch(\kE_{x})\bigr)$ of Gieseker $H$-semistable sheaves on $Y$ with Chern character $\ch(\kE_{x})$, and that this component consists entirely of stable vector bundles. By \'{e}tale descent, the twisted family $\mathcal{E}$ induces a morphism $f\colon X\to M$. More precisely, let $\{U_i\}_{i\in I}$ be an \'etale cover of $X$ trivialising the twist $\alpha_{X}$. Then $\mathcal{E}$ is given by data $(\{\mathcal{E}_i\}_{i\in I},\{\varphi_{ij}\}_{i,j\in I})$, where each $\mathcal{E}_i$ is a locally free sheaf on $U_i\times Y$ and $\varphi_{ij}\colon\mathcal{E}_j|_{U_i\times_X U_j\times Y}\rightarrow\mathcal{E}_i|_{U_i\times_X U_j\times Y}$ are isomorphisms compatible with $\alpha_{X}$. Each $\mathcal{E}_i$ induces a classifying morphism $U_i\rightarrow M$, and hence a morphism $g\colon U':=\coprod_i U_i\rightarrow M$. By \cite[Proposition 3.1.7]{Alper-StacksModuli}, this map descends uniquely to a morphism $f\colon X\rightarrow M$ provided the diagram
   
\[\begin{tikzcd}
	{U'\times_XU'} & {U'} \\
	{U'} & M
	\arrow["{\pi_1}", from=1-1, to=1-2]
	\arrow["{\pi_2}"', from=1-1, to=2-1]
	\arrow["g", from=1-2, to=2-2]
	\arrow["g"', from=2-1, to=2-2]
\end{tikzcd}\]
    commutes. This is easily verified: cover $U'\times_XU'$ by open subsets of the form $U_i\times_X U_j$, and notice that the restriction $\mathcal{E}_i|_{U_i\times_X U_j\times Y}$ induces $g\circ\pi_1|_{U_i\times_X U_j}$, while $\mathcal{E}_j|_{U_i\times_X U_j\times Y}$ induces $g\circ\pi_2|_{U_i\times_X U_j}$. Since $\varphi_{ij}\colon\mathcal{E}_j|_{U_i\times_X U_j\times X}\rightarrow\mathcal{E}_i|_{U_i\times_X U_j\times Y}$ are isomorphisms, then $g\circ\pi_1|_{U_i\times_X U_j}=g\circ\pi_2|_{U_i\times_X U_j}$ and the diagram is commutative.

By \cite[Theorem 3.2.1]{caldararu-thesis}, we have
$0=\Hom\bigl(\Phi_{\kE}(\mathbb{C}(x)),\Phi_{\kE}(\mathbb{C}(x'))\bigr)=\Hom(\mathcal{E}_x,\mathcal{E}_{x'})$ for $x\neq x'$, which implies that $f$ is injective.
Since $\mathcal{E}_x$ is $H$-stable, \cite[Proposition 4.5.2]{HL10} yields
\[
\dim T_{M,[\mathcal{E}_x]}
  = \ext^1(\mathcal{E}_x,\mathcal{E}_x)
  = \ext^1\bigl(\Phi_{\kE}(\mathbb{C}(x)),\Phi_{\kE}(\mathbb{C}(x))\bigr)
  = \ext^1\bigl(\mathbb{C}(x),\mathbb{C}(x)\bigr)
  = \dim T_{X,x},
\]
where the penultimate equality follows from fully faithfulness, and the last equality from \cite[Ex.~11.9 ii)]{Huy03}.
Thus $f$ is injective and the Zariski tangent space to $M$ has dimension $\dim(X)$ at all points of $f(X)$.
The conclusion follows by a standard argument (see \cite[Lemma 1.6]{reede-zhang-smooth-components}).
\end{proof}

\subsection{Vanishing of the Brauer classes}\label{sec_vanishing_twist}
In this Section we provide further details on Construction~\ref{example1} and Theorem~\ref{thm_compatiblevectorbundle}. In particular, in the notation of Theorem~\ref{thm_compatiblevectorbundle}, we give sufficient conditions ensuring that both $\alpha_X=0$ and $\alpha_Y=0$ (see Lemma~\ref{lem_sufficientcondition}).

As in Construction~\ref{example1}, let $S$ be a $K3$ surface of Picard rank $1$ and choose a Mukai vector
\begin{equation}
    u=(r,mh,s)
\end{equation}
such that $r\geq 2$, $\gcd(r,s)=1$, $u^2=0$ and every stable sheaf with Mukai vector $u$ is a stable vector bundle (for this last condition, it is enough to assume $\gcd(r,m)=1$). We recall some well-known facts (see \cite{Mukai2,Yoshioka}).

The moduli space of stable vector bundles with Mukai vector $u$ is a $K3$ surface $M$ of Picard~rank~$1$. We denote by $\hat{h}$ the ample generator of $\Pic(M)$. There exists a universal vector bundle $\mathcal{U}$ on $M\times S$, uniquely determined up to tensoring with the pull-back of a line bundle $L\in \Pic(M)$, such that
\begin{equation}\label{id_c1U}
    c_1(\mathcal{U})=k\hat{h}+mh \in H^2(M,\QQ)\oplus H^2(S,\mathbb{Q}),
\end{equation} for some $k\in \Z$.
There is a rational Hodge isometry
\[
    \rathodge{\U}\colon H^2(M,\QQ)\rightarrow H^2(S,\QQ),
\]
defined by the correspondence defined by the class $\kappa(\U)\sqrt{\td_{M\times S}}$.

Since $\gcd(r,s)=1$, we have $\rathodge{\U^{-1}}(h)=\widehat{h}$ by \cite[Appendix, Example 1]{Yoshioka}.
Moreover, by \cite[Theorem 2.2]{Yoshioka}, the surface $S$ is also a moduli space of vector bundles on $M$ for some Mukai vector $\widehat{u}=(r,k\widehat{h},\widehat{s})$. In particular, the vector $\widehat{u}$ is primitive and isotropic.

\begin{lem}\label{lem_choosek}
    In the above setup, we can choose a universal family $\U$ such that $k=(2\nu-1) m$ for some $\nu\in\Z$.
\end{lem}
\begin{proof}
Since the universal family is determined up to tensoring with the pull-back of a line bundle $L\in \Pic(M)$, it is enough to show that $k\equiv(2\nu-1)m\pmod{r}$ for some $\nu\in\mathbb{Z}$.

Since $\gcd(r,m)=1$, the condition $rs=m^2(h^2/2)$ implies that $r$ divides $h^2/2$. We also have $r\widehat{s}=k^2(\widehat{h}^2/2)=k^2(h^2/2)$ and therefore $\gcd(r,k)$ divides $\widehat{s}=\dfrac{k^2(h^2/2)}{r}=k^2\dfrac{(h^2/2)}{r}$. Thus $\gcd(r,k)$ divides $\gcd(r,k,\widehat{s})=1$, since $\widehat{u}=(r,k\widehat{h},\widehat{s})$ is primitive. It follows that $\gcd(r,k)=1$.

Since $\gcd(r,m)=1$ we can find $x,y\in\ZZ$ such that 
\[
mx+ry=k.
\]
We claim that $x$ can be chosen odd. If $r$ is even, then $\gcd(r,k)=1$ implies that $k$ is odd; hence $x$ is odd. If $r$ is odd and $x$ is already odd, there is nothing to prove. Finally, if $r$ is odd and $x$ is even, then 
\[
m(x+r)+r(y-m)=k,
\]
and we can replace $x$ with $x+r$, which is odd. This proves the claim.
\end{proof}

As explained above, one constructs a vector bundle $\mathcal{U}^{[n]}\in\Coh(M^{[n]}\times S^{[n]})$. There is rational Hodge isometry $\rathodge{\mathcal{U}^{[n]}}:H^2(M^{[n]},\mathbb{Q})\rightarrow H^2(S^{[n]},\mathbb{Q})$: through an important result by Taelman (see \cite{tae23}), it is induced by the correspondence associated with $\kappa(\mathcal{U}^{[n]})\sqrt{\td_{M^{[n]}\times S^{[n]}}}$; we refer to \cite[Section 3.2]{markman} for further details.
By \cite[Section 7]{markman}, we can decompose it as $\rathodge{\mathcal{U}^{[n]}}=(\rathodge{\mathcal{U}},\mathrm{id}):H^2(M,\mathbb{Q})\oplus\QQ\delta_{M^{[n]}}\rightarrow H^2(S,\mathbb{Q})\oplus\QQ\delta_{S^{[n]}}$. By Equations~(7.9) and~(7.10) of \cite{markman}, we have
\begin{equation}
    c_1(\mathcal{U}^{[n]})=a_{S^{[n]}}+a_{M^{[n]}}\in H^2(M^{[n]},\QQ)\oplus H^2(S^{[n]},\QQ),
\end{equation}
where
\begin{align}
    a_{S^{[n]}}&=r^{n-1}n!mh
       - r^n n!\frac{\delta_{S^{[n]}}}{2},\\
    a_{M^{[n]}}&=r^{n-1}n!k\widehat{h}+r^n n!\frac{\delta_{M^{[n]}}}{2}.
\end{align}
It is straightforward to get the following relation:
\begin{equation}\label{eq_identifyingcomponents}
    a_{M^{[n]}}=\frac{k}{m}\rathodge{\mathcal{U}^{[n]}}^{-1}\left(a_{S^{[n]}}\right)+\frac{r^nn!(k+m)}{2m}\delta_{M^{[n]}}
\end{equation}

Moreover, in the setup of Theorem \ref{thm_compatiblevectorbundle}, we have the Brauer classes
\begin{align}
    \alpha_X&=\left[-\dfrac{1}{r^nn!}\eta_X^{-1}\circ\eta_{M^{[n]}}(a_{M^{[n]}})\right]\in\Br(X)=H^2(X,\mathbb{Z})/\Pic(X)\otimes_{\Z}\QQ/\Z \label{eq_twistonX}\\
    \alpha_Y&=\left[-\dfrac{1}{r^nn!}\eta_Y^{-1}\circ\eta_{S^{[n]}}(a_{S^{[n]}})\right]\in\Br(Y)=H^2(Y,\mathbb{Z})/\Pic(Y)\otimes_{\Z}\QQ/\Z \label{eq_twistonY}
\end{align}
By construction, following the notation of Theorem~\ref{thm_compatiblevectorbundle}, following diagram commutes:

\begin{equation}\label{commutation_markings}
    \begin{tikzcd}
	{H^2(X,\QQ)} & {H^2(Y,\QQ)} \\
	{\Lambda_{K3^{[n]},\QQ}} & {\Lambda_{K3^{[n]},\QQ}} \\
	{H^2(M^{[n]},\QQ)} & {H^2(S^{[n]},\QQ)}
	\arrow["{\rathodge{\mathcal{E}}}", from=1-1, to=1-2]
	\arrow["{\eta_X}"', from=1-1, to=2-1]
	\arrow["{\eta_Y}", from=1-2, to=2-2]
	\arrow["\psin", from=2-1, to=2-2]
	\arrow["{\eta_{M^{[n]}}}", from=3-1, to=2-1]
	\arrow["{\rathodge{\mathcal{U}^{[n]}}}"', from=3-1, to=3-2]
	\arrow["{\eta_{S^{[n]}}}"', from=3-2, to=2-2]
\end{tikzcd}
\end{equation}

Combining \eqref{eq_twistonX} with \eqref{eq_identifyingcomponents}, and using the commutativity of diagram~\eqref{commutation_markings}, we obtain
\begin{align}
    \alpha_X&=\left[-\frac{k}{m r^nn!}\rathodge{\mathcal{E}}^{-1}\circ\eta_{Y}^{-1}\circ \eta_{S^{[n]}}(a_{S^{[n]}})-\frac{(k+m)}{2m}\eta_X^{-1}\circ\eta_{M^{[n]}}(\delta_{M^{[n]}})\right]\label{eq_twistX}
\end{align}

A crucial step in the proof of Theorem \ref{thm_main} will be to choose the markings in such a way that the Brauer class $\alpha_Y$ vanishes, so that the kernel $\mathcal{E}$ parametrises untwisted vector bundles. For this purpose, we prove the following 
\begin{lem}\label{lem_sufficientcondition}
    Assume that $\frac{k+m}{2m}\in\Z$ and $\eta_{Y}^{-1}\circ \eta_{S^{[n]}}(a_{S^{[n]}})\in\Pic(Y)_{\mathbb{Q}}$. Then $\alpha_Y=0$ and $\alpha_X=0$.  
\end{lem}
\begin{proof}
    The vanishing of $\alpha_Y$ follows immediately from \eqref{eq_twistonY}. On the other hand, by \eqref{eq_twistX}, the class $\alpha_X$ is represented by $-\frac{k}{mr^nn!}\rathodge{\mathcal{E}}^{-1}\circ\eta_{Y}^{-1}\circ \eta_{S^{[n]}}(a_{S^{[n]}})-\frac{(k+m)}{2m}\eta_X^{-1}\circ\eta_{M^{[n]}}(\delta_{M^{[n]}})$. Since $-\frac{(k+m)}{2m}\eta_X^{-1}\circ\eta_{M^{[n]}}(\delta_{M^{[n]}})\in H^2(X,\Z)$, this has the same Brauer class as $-\frac{k}{m r^nn!}\rathodge{\mathcal{E}}^{-1}\circ\eta_{Y}^{-1}\circ \eta_{S^{[n]}}(a_{S^{[n]}})$. This class is zero because $\rathodge{\mathcal{E}}^{-1}$ is Hodge and $\eta_Y^{-1}\circ \eta_{S^{[n]}}(a_{S^{[n]}})\in \Pic(Y)_\QQ$.
\end{proof}

\begin{rmk}\label{rmk_nicotranquillo}
    By Lemma~\ref{lem_choosek}, we can always assume that $\frac{k+m}{2m}\in\Z$.
\end{rmk}

The Example below shows that the vanishing of $\alpha_Y$ is not enough to conclude automatically that $\alpha_X=0$.

\begin{example}\label{ex_nonfinemoduli}
Let $S$ be a very general polarised $K3$ surface of degree $4$, i.e. $\Pic(S)=\Z h$ and $h^2=4$. Let $M=M_S(u)$ be the moduli $K3$ surface parametrising stable vector bundles with Mukai vector $u=(2,h,1)$. By Lemma~\ref{lem_choosek} and identity (\ref{id_c1U}), we can choose a universal family $\mathcal{U}\in\Coh(M\times S)$, so that $c_1(\mathcal{U})=\widehat{h}+h$.

Choose a class $\gamma\in H^2(S,\Z)$ such that $\gamma^2=0$ and $\gamma\cdot h=1$; this can always be done, as $H^2(S,\Z)$ contains at least a copy of the hyperbolic plane (see Section \ref{sec_k3nlattice}). Set $w:=h-\delta_{S^{[2]}}+2\gamma\in H^2(S^{[2]},\Z)$.
Then $w$ is primitive and $w^2>0$. Fix a marking $\eta_S\colon H^2(S,\mathbb{Z})\rightarrow\Lambda_{K3}$, inducing a marking $\eta_{S^{[2]}}$ of $S^{[2]}$. 
By surjectivity of the period map, there exists a projective marked IHS manifold of $K3^{[2]}$-type $(Y,\eta_Y)$ with period $\sigma_Y\in w^\perp$ and with $\Pic(Y)=\mathbb{Z}\eta_Y^{-1}\circ\eta_{S^{[2]}}(w)$, such that $(S^{[2]},\eta_{S^{[2]}})$ and $(Y, \eta_Y)$ belong to the same connected component of $\mathfrak{M}_{\Lambda_{K3^{[2]}}}$. Combining the previous relations, we obtain
\begin{align*}
    \alpha_Y=\left[-\dfrac{1}{2^22!}\eta_Y^{-1}\circ\eta_{S^{[2]}}(a_{S^{[2]}})\right]=\left[-\frac{1}{2}\eta_Y^{-1}\circ\eta_{S^{[2]}}\left(w\right)+\eta_Y^{-1}\circ\eta_{S^{[2]}}\left(\gamma\right)\right]=0,
\end{align*}
where the last equality follows from the observation that $-\frac{1}{2}\eta_Y^{-1}\circ\eta_{S^{[2]}}\left(w\right)+\eta_Y^{-1}\circ\eta_{S^{[2]}}\left(\gamma\right)\in\Pic(Y)_{\QQ}+H^2(Y,\mathbb{Z})$.
By Lemma \ref{lem_pisurj}, there exists a marked IHS manifold $(X,\eta_X)$ of $K3^\nn-$type such that we can apply Theorem \ref{thm_compatiblevectorbundle} to the pair $(X,\eta_X,Y,\eta_Y)$.
On the other hand, we obtain 
\begin{align*}
\alpha_X
&= \left[-\frac{1}{8}\eta_X^{-1}\circ\eta_{M^{[2]}}\circ\rathodge{\mathcal{U}^{[2]}}^{-1}\left(a_{S^{[2]}}\right)-\eta_X^{-1}\circ\eta_{M^{[2]}}\left(\delta_{M^{[2]}}\right)\right] 
&& \text{by Equation~\eqref{eq_twistX}} \\
&= \left[-\frac{1}{8}\eta_X^{-1}\circ\eta_{M^{[2]}}\circ\rathodge{\mathcal{U}^{[2]}}^{-1}\left(a_{S^{[2]}}\right)\right]
&& \text{since } \eta_X^{-1}\circ\eta_{M^{[2]}}\left(\delta_{M^{[n]}}\right)\in H^2(X,\Z) \\
&= \left[-\frac{1}{2}\rathodge{\mathcal{E}}^{-1}\circ\eta_Y^{-1}\circ\eta_{S^{[2]}}\left(w\right)+\eta_X^{-1}\circ\eta_{M^{[2]}}\circ\rathodge{\mathcal{U}^{[2]}}^{-1}(\gamma)\right]
&& \text{by the definitions of $\gamma$ and $w$} \\
&= \left[\eta_X^{-1}\circ\eta_{M^{[2]}}\circ\rathodge{\mathcal{U}^{[2]}}^{-1}(\gamma)\right]
&& \text{since }\frac{1}{2}\rathodge{\mathcal{E}}^{-1}\circ\eta_Y^{-1}\circ\eta_{S^{[2]}}\left(w\right)\in\Pic(X).
\end{align*}

\noindent We now show that this class is non-zero in $\Br(X)$. Suppose, for contradiction, that $\alpha_X=0$. Then
\begin{equation*}
    \eta_X^{-1}\circ\eta_{M^{[2]}}\circ\rathodge{\mathcal{U}^{[2]}}^{-1}(\gamma)\in\Pic(X)_\QQ+H^2(X,\Z).
\end{equation*}
Applying $\rathodge{\mathcal{U}^{[2]}}\circ\eta_{M^{[2]}}^{-1}\circ\eta_X$, and using the description of $\Pic(X)$, we obtain
\[
   \gamma\in \rathodge{\mathcal{U}^{[2]}}\bigl(H^2(M^{[2]},\Z)\bigr)+\QQ w.
\]
Hence there exists $\lambda\in\QQ$ such that
\begin{equation*}
   \gamma-\lambda w=(1-2\lambda)\gamma-\lambda h+\lambda\delta_{S^{[n]}}\in \rathodge{\mathcal{U}^{[2]}}\bigl(H^2(M^{[2]},\Z)\bigr).
\end{equation*}
\noindent Since $\rathodge{\mathcal{U}^{[2]}}=(\rathodge{\mathcal{U}},\mathrm{id})$, it follows that $\lambda\in\Z$.
We get
\begin{equation}\label{eq:odd}
   h\cdot\bigl((1-2\lambda)\gamma-\lambda h+\lambda\delta_{S^{[n]}}\bigr)
   =(1-2\lambda)h\cdot\gamma-\lambda h^2
   =1-2\lambda-4\lambda
   =1-6\lambda.
\end{equation}
This is odd for every $\lambda\in\Z$. On the other hand, by the definition of $\rathodge{\mathcal{U}}$, for every $\beta\in H^2(S,\mathbb{Z})$ one has
\begin{equation*}
    \beta\in \rathodge{\mathcal{U}}(H^2(M,\mathbb{Z}))\iff h\cdot\beta\in 2\mathbb{Z},
\end{equation*}
contradicting \eqref{eq:odd}. This proves $\alpha_X\neq 0$.
\end{example}
\begin{rmk}\label{rmk_nonfinemoduli}
    In the setup of Example~\ref{ex_nonfinemoduli}, combining Lemma~\ref{lem_fibersarestable} and Lemma~\ref{lem_omegaimpliesH}, the fibers of $\mathcal{E}$ are $H$-slope stable with respect to a suitable polarisation $H\in\Pic(Y)$ (see the proof of Theorem~\ref{thm_main} for further details). 
    By Lemma~\ref{lem_Ymodulisp}, the previous Example~\ref{ex_nonfinemoduli} produces a non-fine moduli space $X$ of vector bundles on a projective IHS fourfold $Y$. 
    Moreover we get a twisted derived equivalence $\D(X,\alpha_X^{-1})\to \D(Y)$ for a non trivial Brauer class $\alpha_X$. See Remark~\ref{rmk_iso_nonfine} for further discussions on this example.
\end{rmk}

\section{Paths in the moduli space of marked IHS manifolds}\label{sec_admissible}
This Section is devoted to the study of  the connected components of the moduli space $\mathfrak{M}_\Lambdan$ of marked IHS manifolds of $K3^{[n]}$-type. In the setup of Theorem~\ref{thm_compatiblevectorbundle}, the deformation of projectively hyperholomorphic vector bundles takes place within a fixed connected component of the moduli space of marked pairs, along generic diagonal twistor paths. Consequently, we need to determine whether two marked IHS manifolds of $K3^{[n]}$-type $(X_1,\eta_{X_1})$ and $(X_2,\eta_{X_2})$ belong to the same connected component of $\mathfrak{M}_{\Lambdan}$ or, equivalently, can be connected by parallel transport (see Remark~\ref{rmk_components}). To address this problem, the main tools we employ are Eichler's criterion, which describes the orbits of primitive elements in the $K3^{[n]}$-lattice under a certain class of isometries (the so-called Eichler transvections), and Markman's classification of monodromy operators for IHS manifolds of $K3^{[n]}$-type. The core result of the Section is Lemma~\ref{markingexistence}, which provides a sufficient criterion for a IHS manifold $Y$ to admit a marking $\eta_Y$ so that $(Y,\eta_Y)$ lies in the same connected component of $\mathfrak{M}_\Lambdan$ as a marked Hilbert scheme $(S^{[n]},\eta_{S^{[n]}})$, to which successfully apply the study carried out in Section \ref{sec_vanishing_twist}. This will be the starting point to perform Construction~\ref{example1}.

\subsection{Eichler's criterion} We now recall the notation and definitions needed to prove Lemma~\ref{lem_latticetheory}, which is an application of Eichler's criterion to the $K3^{[n]}$-lattice.

Let $L$ be a non-degenerate lattice. For any element $w\in L$, we define its \textit{divisibility} $\operatorname{div}(w)$ to be the unique positive generator of the ideal
\[
(w,L) := \{w\cdot \lambda \mid \lambda \in L\} \subset \ZZ.
\]
The non-degeneracy of the symmetric form induces an injective group homomorphism $ i \colon L \to L^\vee$. The quotient $A_L := L^\vee / L$ is a finite group, called the \textit{discriminant group}. The lattice $L$ is said to be \textit{unimodular} if $A_L=0$.
\subsubsection{Remarkable subgroups of $\mathrm{O}(L)$}
Any isometry $g \in \mathrm{O}(L)$ induces a group automorphism $\overline{g} \in \mathrm{Aut}(A_L)$. We denote by $\widetilde{\mathrm{O}}(L)$ the subgroup of $\mathrm{O}(L)$ consisting of those isometries whose action on the discriminant group $A_L$ is equal to $\id_{A_L}$. On the other hand, we define the group $\mathrm{O}^+(L)$ of \textit{orientation-preserving isometries} of $L$ as the isometries $g \in \mathrm{O}(L)$ that preserve the orientation of a maximal positive definite subspace of $L \otimes \RR$ (see \cite[Section~4]{Mar11}). This allows us to introduce
$\widetilde{\mathrm{O}}^+(L) := \widetilde{\mathrm{O}}(L) \cap \mathrm{O}^+(L) \subseteq \mathrm{O}(L)$, which plays a key role in Proposition~\ref{prop_eichler}. Recall that the unimodular \textit{hyperbolic plane} is the lattice
\begin{center}
    $U = \Z e \oplus \Z f$, with $e^2 = f^2 = 0$ and $(e,f) = 1$.
\end{center}

\begin{prop}[\textit{Eichler's criterion, \cite[Proposition~3.3(i)]{gritsenko-hulek-sankaran-abelianisation}}]\label{prop_eichler}
Let $L$ be an even lattice, such that $L=U_1\oplus L_1$ and $L_1=U_2\oplus L_2$, where $U_1$ and $U_2$ denote, respectively, two distinguished copies of the unimodular hyperbolic plane in $L$. Then, if $v_1,v_2\in L$ are primitive elements such that
\begin{enumerate}
    \item [{\normalfont{(1)}}] $v_1^2=v_2^2$,
    \item [{\normalfont{(2)}}] $[v_1/\mathrm{div}(v_1)]=[v_2/\mathrm{div}(v_2)]\in A_L$,
\end{enumerate}
there exists $\tau\in \widetilde{\mathrm{O}}^{+}(L)$ such that $\tau(v_1)=v_2$.
\end{prop}

Finally, we introduce the \textit{Weyl group of reflections} defined by
\begin{equation}
    \label{weylgrp}
    \mathsf{W}(L) := \{ g \in \mathrm{O}^+(L) \mid \overline{g} = \pm \id_{A_L} \},
\end{equation} which is an index~$2$ extension of $\widetilde{\mathrm{O}}^+(L)$. We refer to \cite[Lemma~4.10]{Mar08} for a further characterisation of $\mathsf{W}(L)$.
\subsubsection{The $K3^{[n]}$-lattice}\label{sec_k3nlattice}

We recall that, if $S$ is a K3 surface, then $\H^2(S,\Z)$, equipped with the intersection pairing, is abstractly isomorphic to the unimodular lattice $\Lambda_{K3}:= U^{\oplus 3}\oplus E_8(-1)^{\oplus 2}$ of signature $(3,19)$, where $E_8(-1)$ is the unique even, unimodular, negative definite lattice of rank $8$. We will refer to $\Lambda_{K3}$ as the $K3$-\textit{lattice}. If $X$ is an IHS manifold of $K3^{[n]}$-type, with $n\geq 2$, then its BBF-lattice is abstractly isomorphic to the even lattice $\Lambdan:=\Lambda_{K3}\oplus \brak{-2(n-1)}$ of signature $(3,20)$ and cyclic discriminant of order $2(n-1)$, where $\brak{-2(n-1)}$ is the rank $1$ lattice generated by an element of square $-2(n-1)$. We will refer to $\Lambdan$ as the \textit{$\kn$-lattice}.
\begin{rmk}
    The $K3^{[n]}$-lattice satisfies the hypotheses of Proposition~\ref{prop_eichler}.
\end{rmk}
We choose $\delta$ to be a generator of $\brak{-2(n-1)}$ and fix, once and for all, an identification $\Lambdan = \Lambda_{K3} \oplus \ZZ \delta$. Then, every $w\in\Lambdan$ uniquely decomposes as $w=w_{K3}+m_w\delta$, where $m_w$ is an integer and $w_{K3}\in\Lambda_{K3}$. 
\begin{rmk}\label{rmk_primitiveimpliescoprime}
Note that $w$ is primitive if and only if $\gcd(m_w,\mathrm{div}(w))=1$. This follows by observing that, since $\mathrm{div}(w)\mid \mathrm{div}_{\Lambda_{K3}}(w_{K3})$ and $\Lambda_{K3}$ is unimodular, we can write
\[
w = \mathrm{div}(w)\, w'_{K3} + m_w \delta
\]
for some $w'_{K3}\in \Lambda_{K3}$. 
In particular we have that $\dv(w)$ divides $\delta^2=-2(n-1)$.
\end{rmk}
\begin{lem}\label{lem_latticetheory}
Let $w=w_{K3}+m_w\delta\in\Lambdan=\Lambda_{K3}\oplus\ZZ\delta$ be a primitive element. Then, for every integer $m\equiv \pm m_w \pmod{\mathrm{div}(w)}$, there exists an isometry $g\in\W(\Lambdan)$ such that $g(w)=\mathrm{div}(w)v' + m\delta$, for some primitive $v'\in\Lambda_{K3}$.
\end{lem}
\begin{proof}
    It suffices to consider the case $m \equiv m_w \pmod{\mathrm{div}(w)}$, since in the case $m \equiv -m_w \pmod{\mathrm{div}(w)}$ we can post-compose with the map
    \begin{align*}
        g\colon \Lambda_{K3}\oplus\mathbb{Z}\delta &\longrightarrow \Lambda_{K3}\oplus\mathbb{Z}\delta \\
            u_{K3}+m\delta &\longmapsto u_{K3}-m\delta
    \end{align*}
    which belongs to $\W(\Lambdan)$. Choose $U=\ZZ e\oplus \ZZ f\subset\Lambda_{K3}$ a hyperbolic plane, so that $e^2=f^2=0$ and $        (e,f)=1$. 
    Since $\widetilde{\mathrm{O}}^{+}(L)\subset\mathsf{W}(L)$, the result follows from Proposition~\ref{prop_eichler} applied to $v_1=w$ and $v_2=\mathrm{div}(w)v' + m\delta$ where
    \begin{equation}\label{v}
        v':=e+\left(\frac{w^2-m^2\delta^2}{2\mathrm{div}(w)^2}\right)f.
    \end{equation}
    We now verify the required conditions:
    \begin{enumerate}
        \item We show that the coefficient of $f$ in \eqref{v} is an integer, showing that $v'$ is a well defined element of $\Lambda_{K3}$. In fact, as explained in Remark~\ref{rmk_primitiveimpliescoprime}, we have a decomposition $w=\mathrm{div}(w)w'_{K3}+m_w\delta$ which yields
        \[
            \frac{w^2-m^2\delta^2}{2\mathrm{div}(w)^2}=\frac{\mathrm{div}(w)^2(w'_{K3})^2+(m_w^2-m^2)\delta^2}{2\mathrm{div}(w)^2}.
        \]
        Both summands in the numerator are divisible by $2\mathrm{div}(w)^2$. Indeed, it holds $2\mathrm{div}(w)^2 \mid \mathrm{div}(w)^2 (w'_{K3})^2$, and also $2\mathrm{div}(w)^2 \mid (m_w^2 - m^2)\delta^2$, since $\mathrm{div}(w) \mid (m_w - m)$ by the definition of $m$, and $2\mathrm{div}(w) \mid (m_w + m)\delta^2$: if $\mathrm{div}(w)$ is odd, then $2\mathrm{div}(w) \mid \delta^2$, while if $\mathrm{div}(w)$ is even, then $2 \mid (m_w + m)$ and $\mathrm{div}(w) \mid \delta^2$.
            \item The element $v'$ is primitive by definition (see \eqref{v}), hence $v_2$ is primitive by Remark~\ref{rmk_primitiveimpliescoprime}, since the primitivity of $w$ implies $\gcd(m_w,\mathrm{div}(w)) = 1$ and consequently $\gcd(m,\mathrm{div}(w)) = 1$.
            \item It is straightforward to check that $v_1^2 = v_2^2$.
            \item Finally, we have an equality $[v_1/\mathrm{div}(v_1)]=[v_2/\mathrm{div}(v_2)]\in A_\Lambdan$. The coefficients of $e$ and $f$ in $v_2$ are both multiples of $\mathrm{div}(w)$ and $(v_2,f)=\mathrm{div}(w)$, so we have $\mathrm{div}(v_2)=\mathrm{div}(w)$. Thus, we obtain \begin{align*}
                 \left[\frac{v_1}{\mathrm{div}(v_1)}\right]- \left[\frac{v_2}{\mathrm{div}(v_2)}\right]&=\left[\frac{w-v_2}{\mathrm{div}(w)}\right]=\left[\frac{\mathrm{div}(w)w'_{K3}+m_w\delta-\mathrm{div}(w)v'-m\delta}{\mathrm{div}(w)}\right]\\
                &=\left[\frac{\mathrm{div}(w)w'_{K3}+(m_w-m)\delta-\mathrm{div}(w)v'}{\mathrm{div}(w)}\right]=0,
            \end{align*}
            where the last equality follows from the fact that all summands in the numerator are divisible by $\mathrm{div}(w)$. \qedhere
        \end{enumerate}
    \end{proof}

\subsection{Monodromy and moduli of IHS manifolds of $K3^{[n]}$-type}
Let us recall that, given a deformation $p\colon \mathcal{X}\to T$ of IHS manifolds, we can define an isometry between the BBF lattices of each pair of fibers as parallel transport operator induced by any continuous path $\gamma$ in $T$ connecting the two respective base points. For any IHS manifold $X$, we denote by $\mon^2(X)$ its \textit{monodromy group}, generated by parallel transport operators from $X$ to itself constructed as above (see \cite{Mar11}).

\begin{rmk}\label{rmk_components}
    The monodromy group $\mon^2(X)$ of an IHS manifold $X$ is a deformation invariant and a finite index subgroup of $\OO^+(\H^2(X,\Z))$ (see \cite[Theorem 3.4]{Ver09} and \cite[Section 4]{Mar11}). Moreover, it plays a key role in the characterisation of the connected components of the moduli space $\mathfrak{M}_\Lambda$ of marked IHS manifolds, for any suitable lattice $\Lambda$. Indeed, the group $\OO(\Lambda)$ acts on each connected component of $\mathfrak{M}_\Lambda$ interchanging the markings, and the monodromy group stabilizes each connected component under this action. More precisely, by \cite[Lemma 7.5]{Mar11}, two marked IHS manifolds $(X_1,\eta_1)$ and $(X_2,\eta_2)$ belong to the same connected component of $\mathfrak{M}_\Lambda$ if and only if $\eta_2^{-1}\circ \eta_1$ is a parallel transport operator.
\end{rmk}
 The monodromy group has been computed for all known deformation classes of IHS manifolds. In the case of our interest, we have the following description.

\begin{thm}[\textit{\cite[Corollary 1.8]{Mar08}, \cite[Theorem 2.2]{Mar10}}]\label{markman_monodromy} Let $X$ be an IHS manifold of $K3^{[n]}$-type, with $n\geq 2$. Then $$\mon^2(X)=\W(\H^2(X,\Z)).$$
\end{thm}

\begin{rmk}\label{rmkbeauville}
    Let $g\colon \mathcal{S}\rightarrow B$ be a smooth and proper family of $K3$ surfaces, and consider the smooth and proper family $\bar{g}\colon \mathcal{X}\rightarrow B$ given by the relative Douady space of $n$-points associated to $g$. Then, any parallel transport defined in the family $\bar{g}$ preserves the canonical decomposition $\H^2(S^{[n]},\mathbb{Z})=\H^2(S,\mathbb{Z})\oplus\mathbb{Z}\delta_{S^{[n]}}$, by \cite[Section 9, Lemme 2]{Bea83}.
\end{rmk}

\begin{rmk}
    The geometric decomposition of $H^2(S^\nn,\Z)$ in Remark \ref{rmkbeauville} provides a natural way to define marked pairs of hyperkähler manifolds of $K3^{[n]}$-type starting from a marked K3 surface. Indeed, if $\eta\colon \H^2(S,\Z)\to \Lambda_{K3}$ is a marking on a projective K3 surface $S$, then the unique isometry $\eta_{S^{[n]}}\colon \H^2(S^{[n]},\Z)\to \Lambdan$ restricting to $\eta$ on $\H^2(S,\Z)$ and sending $\delta$ to the generator of $\brak{-2(n-1)}$ defines a marking for $S^{[n]}$. If we let $\mathfrak{M}_{K3}^0$ be the connected component of $\mathfrak{M}_{K3}$ containing the pair $(S,\eta)$ and $\mathfrak{M}_\Lambdan^0$ be the connected component of $\mathfrak{M}_\Lambdan$ containing the pair $(S^{[n]},\eta_{S^{[n]}})$, we get a well defined map \begin{align*}
        \theta \colon \mathfrak{M}_{K3}^0&\to \mathfrak{M}_\Lambdan^0\\
        (S',\eta') & \to (S'^{[n]},\eta'^{[n]}).
    \end{align*} By Remark \ref{rmkbeauville}, the map $\theta$ is compatible with the respective period maps. Indeed, the embedding $\Lambda_{K3}\hookrightarrow \Lambdan$ induces a well defined map $\tau\colon D_{K3}\to D_\Lambdan$ between the respective period domains. If we let $P\colon \mathfrak{M}_{K3} \to D_{K3}$ and $P_\Lambdan\colon \mathfrak{M}_{\Lambdan}^0\to D_\Lambdan^0$ be the restrictions of the respective global period maps, then we get the following commutative diagram \begin{equation}\label{diag_markedK3tomarkedk3n}
\begin{tikzcd}
	{\mathfrak{M}_{K3}^0} & {\mathfrak{M}_{\Lambdan}^0} \\
	{D_{K3}} & {D_\Lambdan}
	\arrow["\theta", from=1-1, to=1-2]
	\arrow["{P_{K3}}"', from=1-1, to=2-1]
	\arrow["P_\Lambdan", from=1-2, to=2-2]
	\arrow["\tau", from=2-1, to=2-2].
\end{tikzcd}
\end{equation}
\end{rmk}
    
Combining Lemma~\ref{lem_latticetheory} with Markman’s description of the monodromy group of IHS manifolds of $K3^{[n]}$-type (see Theorem \ref{markman_monodromy}), we obtain the main result of this Section, which will be used to choose suitable markings in Section~\ref{sec_mainthm}.
\begin{lem}\label{lem_monodromy}
     Let $S$ be a $K3$ surface and let $\delta_{S^{[n]}}$ be half the class of the exceptional divisor on $S^{[n]}$. Let $v=v_S+m_v\delta_{S^{[n]}}\in \H^2(S^{[n]},\mathbb{Z})=\H^2(S,\mathbb{Z})\oplus\mathbb{Z}\delta_{S^{[n]}}$ be a primitive vector and let $m\in\mathbb{Z}$ be an integer such that $m \equiv \pm m_v \pmod{\mathrm{div}(v)}$. 
     Then, there exists a monodromy operator $g\colon \H^2(S^{[n]}, \mathbb{Z})\rightarrow \H^2(S^{[n]}, \mathbb{Z})$ such that $g(v)=\mathrm{div}(v) v'+m\delta_{S^{[n]}}$ with $v'\in \H^2(S,\mathbb{Z})$ primitive.
\end{lem}
\begin{proof}
    By Lemma~\ref{lem_latticetheory}, there exists $g\in \W(\H^2(S^{[n]},\ZZ))$ such that $g(v)=\mathrm{div}(v)v'+m\delta_{S^{[n]}}$. The result follows from $\W(\H^2(S^{[n]},\mathbb{Z}))=\mathrm{Mon}^2(S^{[n]})$ (see Theorem~\ref{markman_monodromy}). 
\end{proof}

\subsection{An invariant of primitive classes}
Let $Y$ be a IHS manifold of $K3^{[n]}$-type. By definition, there exists a $K3$ surface $S$ and a parallel transport operator $g\colon \H^2(Y,\mathbb{Z})\rightarrow \H^2(S^{[n]},\mathbb{Z})$. Recall that we have a canonical decomposition $\H^2(S^{[n]},\mathbb{Z})=\H^2(S,\mathbb{Z})\oplus\mathbb{Z}\delta_{S^{[n]}}$. Hence, for every primitive $w\in \H^2(Y,\mathbb{Z})$, we can write $g(w)=v+m\delta$, with $v\in \H^2(S,\mathbb{Z})$. By Remark~\ref{rmk_primitiveimpliescoprime}, we have $\gcd(m,\dv(w))=1$. Therefore, the class $[m]$ defines an element of $(\mathbb{Z}/\dv(w)\mathbb{Z})^*$ and we can consider its image in the quotient of multiplicative groups  $(\mathbb{Z}/\dv(w)\mathbb{Z})^*/\{\pm1\}$.

\begin{rmk}
    We adopt the convention that for $d=1$ the group $(\mathbb{Z}/d\mathbb{Z})^*/\{\pm1\}$ is the trivial group $\{1\}$. In particular, if $\dv(w)=1$, then every integer $m$ defines the same class in $(\mathbb{Z}/\dv(w)\mathbb{Z})^*/\{\pm1\}$.
\end{rmk}

\begin{lem}\label{lem_independencyofparalleltransport}
        The class $[m]\in(\mathbb{Z}/\dv(w)\mathbb{Z})^*/\{\pm1\}$ does not depend on the $K3$ surface $S$ or on the chosen parallel transport.
\end{lem}
\begin{proof}
    Consider, for $i=1,2$, two $K3$ surfaces $S_i$ and parallel transports
    \[
        g_i\colon \H^2(Y,\mathbb{Z})\rightarrow \H^2(S_i^{[n]},\mathbb{Z}).
    \]
    In order to simplify the notation, we set $\delta_1:=\delta_{S_1^{[n]}}$ and $\delta_2:=\delta_{S_2^{[n]}}$. Notice that we have canonical decompositions $g_1(w)=w_1+m_1\delta_1$ and $g_2(w)=w_2+m_2\delta_2$.
    By the description of the discriminant group for the $K3^{[n]}$-lattice, we obtain that
    \[
    \left[\frac{g_i(w)}{\mathrm{div}(w)}\right]=\left[\frac{m_i\delta_i}{\mathrm{div}(w)}\right]\in A_{\H^2(S_i^{[n]},\mathbb{Z})}
    \]
    for $i=1,2$.
    The key fact is that, by Remark \ref{rmk_components}, there is a parallel transport $h\colon \H^2(S_1,\mathbb{Z})\rightarrow \H^2(S_2,\mathbb{Z})$ which induces a parallel transport $h'\colon \H^2(S_1^{[n]},\mathbb{Z})\rightarrow \H^2(S_2^{[n]},\mathbb{Z})$, which preserves the canonical decompositions.
    In other words, we have $h'(\H^2(S_1,\mathbb{Z}))=\H^2(S_2,\mathbb{Z})$ and $h'(\delta_1)=\delta_2$.
    Now consider the composition $f:=h'\circ g_1\circ g_2^{-1}$.
    By definition of those maps we get 
    \[f\left[\dfrac{m_2\delta_2}{\dv(w)}\right]=h'\circ g_1\circ g_2^{-1}\left[\dfrac{m_2\delta_2}{\dv(w)}\right]=h'\left[\dfrac{m_1\delta_1}{\dv(w)}\right]=\left[\dfrac{m_1\delta_2}{\dv(w)}\right].\]
    On the other hand, the map $f$ is a monodromy operator by construction, hence, by Theorem \ref{markman_monodromy}, it acts as $\pm 1$ on the discriminant group, yielding 
    \[f\left[\dfrac{m_2\delta_2}{\dv(w)}\right]=\pm\left[\dfrac{m_2\delta_2}{\dv(w)}\right].\]
    We conclude that 
    \[0=\left[\dfrac{m_2\delta_2\pm m_1\delta_2}{\dv(w)}\right]=\left[\dfrac{(m_2\pm m_1)\delta_2}{\dv(w)}\right]\]
    thus, since $\delta_2$ primitive, the claim holds.
    \end{proof}

    Lemma~\ref{lem_independencyofparalleltransport} shows that the following definition is independent of the choice of parallel transport.
\begin{df}\label{def_invariant}
Let $w \in \H^2(Y,\mathbb{Z})$  be a primitive class, and let $g\colon \H^2(Y,\mathbb{Z})\rightarrow \H^2(S^{[n]},\mathbb{Z})$ be a parallel transport operator to a Hilbert scheme, such that $g(w)=v+m\delta_{S^{[n]}}$ under the natural decomposition. We denote by $\ell(w)$ the class $[m]\in(\ZZ/\dv(w)\ZZ)^*/\{\pm 1\}$, where the latter is a quotient of multiplicative groups.
\end{df}

The following is a straightforward consequence of Lemma~\ref{lem_monodromy}.
\begin{lem}\label{lem_paralleltransporttohilb}
        Let $w\in \H^2(Y,\mathbb{Z})$ be primitive and suppose that $m$ is an integer such that $[m]= \ell(w)$ in $(\ZZ/\dv(w)\ZZ)^*/\{\pm 1\}$. Then there exists a $K3$ surface $S$ and a parallel transport $g\colon \H^2(Y,\mathbb{Z})\rightarrow \H^2(S^{[n]},\mathbb{Z})$ such that $g(w)=\mathrm{div}(w) w'+m\delta_{S^{[n]}}$ with $w'\in \H^2(S,\mathbb{Z})$ primitive.
    \end{lem}

    We conclude the Section with a geometric application of Lemma \ref{lem_paralleltransporttohilb} that will be crucial in the next constructions. Fix the usual identification $\Lambdan=\Lambda_{K3}\oplus\mathbb{Z}\delta$.
        \begin{lem}\label{markingexistence}
            Let $Y$ be an IHS manifold of $K3^{[n]}$-type. Let $w\in \H^2(Y,\mathbb{Z})$ be primitive and $m$ be a positive integer such that $w^2>-m^2(2n-2)$ and  $[m]=\ell(w)$ in $(\ZZ/\dv(w)\ZZ)^*/\{\pm 1\}$. Then there exists a marking $\eta_Y\colon \H^2(Y,\mathbb{Z})\rightarrow\Lambdan$ and $(S,\eta_S)\in\mathfrak{M}_{\Lambda_{K3}}$ a marked $K3$ surface such that
            \begin{enumerate}
                \item [{\normalfont{(1)}}] $(Y,\eta_Y)$ and $(S^{[n]},\eta_{S^{[n]}})$ lie in the same connected component of $\mathfrak{M}_\Lambdan$, where $\eta_{S^{[n]}}$ is the natural marking induced by $\eta_S$;
                \item [{\normalfont{(2)}}] $\eta_Y(w)=\dv(w)v+m\delta\in\Lambdan$, where $v\in\Lambda_{K3}$ is primitive.
                \item [{\normalfont{(3)}}] $\Pic(S)\cong \Z\eta_S^{-1}(v)$.
            \end{enumerate}
        \end{lem}
        \begin{proof}
    By Lemma~\ref{lem_paralleltransporttohilb}, there exists a $K3$ surface $S'$ and a parallel transport operator $g\colon \H^2(Y,\mathbb{Z})\rightarrow \H^2(S'^{[n]},\mathbb{Z})$ such that $g(w)=\mathrm{div}(w) w'+m\delta_{S'^{[n]}}$ with $w'\in \H^2(S',\mathbb{Z})$ primitive. Choose any marking $\eta_{S'}$ on the $K3$ surface $S'$ and define $v:=\eta_{S'}(w')$. Then, we can consider $(S'^{[n]},\eta_{S'^{[n]}}):=\theta(S',\eta_{S'})$, following the notation of Diagram~\ref{diag_markedK3tomarkedk3n}. By assumption $w^2+m^2(2n-2)>0$, we obtain $(w')^2>0$ and hence $v^2>0$.
    A consequence of the surjectivity of the period map for $K3$ surfaces (see \cite[Corollary 14.3.1]{HuyK3}), there is some $(S,\eta_S)$ in the same connected component of $(S',\eta_{S'})$ such that $\Pic(S)\cong \Z\eta_S^{-1}(v)$. Then $(S^{[n]},\eta_{S^{[n]}}):=\theta(S,\eta_S)$ is in the same connected component as $(S'^{[n]},\eta_{S'^{[n]}})$.
    Therefore, the result follows by setting $\eta_Y:=\eta_{S'^{[n]}}\circ g$. This implies that $(Y,\eta_Y)$ and $(S'^{[n]},\eta_{S'^{[n]}})$ are in the same connected component and therefore also $(Y,\eta_Y)$ and $(S^{[n]},\eta_{S^{[n]}})$ are.
\end{proof}

For completeness, we include here a Lemma describing the orbit of the $K3$-part under monodromy, once the coefficient of $\delta_{S^{[n]}}$ is fixed.
\begin{lem}
    Let $S$ be a $K3$ surface. Let $w=w_{K3}+m\delta_{S^{[n]}}\in \H^2(S^{[n]},\mathbb{Z})=\H^2(S,\mathbb{Z})\oplus\mathbb{Z}\delta_{S^{[n]}}$ be primitive. Pick $w'_{K3}\in \H^2(S,\mathbb{Z})$. Then there exists $g\in \mathrm{Mon}^2(S^{[n]})$ such that $g(w)=w'_{K3}+m\delta_{S^{[n]}}$ if and only if the following three conditions hold:
    \begin{enumerate}
        \item [{\normalfont{(1)}}] $(w'_{K3})^2=(w_{K3})^2$;
        \item [{\normalfont{(2)}}] $\gcd(\dv(w'_{K3}),m)=1$;
        \item [{\normalfont{(3)}}] $\dv(w)=\gcd(\dv(w'_{K3}),2n-2)$.
    \end{enumerate}
\end{lem}
\begin{proof}
    The conditions are easily seen to be necessary, since a monodromy operator preserves squares, primitivity and divisibility. They are also sufficient, by Eichler's criterion Proposition~\ref{prop_eichler} .
\end{proof}

\section{Moduli spaces of modular vector bundles}\label{sec_mainthm}
    In this Section, we present the proof of our main result.
\begin{thm}\label{thm_main}
Let $n\geq 2$ be an integer, let $Y$ be a projective IHS manifold of $K3^{[n]}$-type and let $w \in \NS(Y)$ be a primitive class. If $\dv(w)$ is odd, consider a positive even integer $r$ such that : \begin{enumerate}[label=(\alph*)]
    \item [\mylabel{a}{\normalfont{(a)}}]  $\left[\frac{r}{2}\right]=\ell(w)\in (\Z/\dv(w)\Z)^*/\{\pm 1\}$;
    \item [\mylabel{b}{\normalfont{(b)}}] $\frac{w^2}{2r}+ \frac{r}{4}(n-1)$ is a positive integer;
    \item [\mylabel{c}{\normalfont{(c)}}] $\frac{w^2}{2r}+ \frac{r}{4}(n-1)$ is coprime with $r$.
\end{enumerate}
If $\dv(w)$ is even, consider a positive integer $r\geq 3$ such that:
\begin{enumerate}[label=(\alph*')]
    \item [\mylabel{a2}{\normalfont{(a')}}]   $[r]=\ell(w)\in (\Z/\dv(w)\Z)^*/\{\pm 1\}$;
    \item [\mylabel{b2}{\normalfont{(b')}}] $\frac{w^2}{2r} + r(n-1)$ is a positive integer;
    \item [\mylabel{c2}{\normalfont{(c')}}] $\frac{w^2}{2r}$ is coprime with $r$.
\end{enumerate}
Then, there exists a polarisation $H$ on $Y$ and a connected component $X$ of the moduli space $M_{Y,H}$ of Gieseker $H$-semistable sheaves on $Y$, such that 
\begin{enumerate}[label=$\bullet$]
    \item $X$ is an IHS manifold of $K3^{[n]}$-type.

    \item $X$ parametrises $\mu_H$-stable modular vector bundles $E$ on $Y$ with the following invariants:

    \begin{minipage}[l]{0.40\textwidth}
    \begin{equation*}
        \cc_1(E)= \bigg\{\begin{array}{ll}
           r^{n-1}n! w  & \textit{if }\dv(w) \textit{ is odd} \\
           r^{n-1}n! \frac{w}{2}  & \textit{if }\dv(w) \textit{ is even}
        \end{array}
    \end{equation*}
\end{minipage}
\hspace{1ex}
\begin{minipage}[l]{0.35\textwidth}
    \begin{align*}
       \rk(E)&= r^n n!\\
        \Delta(E) &= \frac{(r^n n!)^2}{12}\cc_2(Y).\end{align*}
\end{minipage}
\vspace{0.5ex}

    \item The Fourier-Mukai transform $$\Phi_\kE \colon \D(X) \to \D(Y)$$ with kernel $\kE$ is an equivalence, where $\mathcal{E}$ is a  universal family on $X\times Y$.

\end{enumerate}    
\end{thm}
\begin{proof}[Proof of Theorem \ref{thm_main} for $\dv(w)$ odd]
The idea of the proof is to construct the universal family $(\kE,\alpha_\kE)$ on $X \times Y$, as in Theorem~\ref{thm_compatiblevectorbundle}, in such a way that the twist vanishes. Thus $\kE$ is a family of untwisted vector bundles on $Y$. We then prove that the fibers over points $x \in X$ are stable with respect to a suitable polarisation and show that the induced map from $X$ to the moduli space is an isomorphism onto a connected component.

By \ref{a} and \ref{b} we have $\ell(w)=\left[r/2\right]$ and $w^2 > -\frac{r}{4}(n-1)\cdot 2r = -(r/2)^2(2n-2)$, hence we can apply Lemma~\ref{markingexistence} with $m = r/2$. Therefore, there exists a marked Hilbert scheme $(S^{[n]},\eta_{S^{[n]}})$ and a marking $\eta_Y\colon \H^2(Y,\mathbb{Z})\rightarrow \Lambdan$, such that
\begin{enumerate}[label=(\arabic*)]
    \item $(Y,\eta_Y)$ and $(S^{[n]},\eta_{S^{[n]}})$ lie in the same connected component of $\mathfrak{M}_\Lambdan$, which we denote by $\mathfrak{M}^\circ_\Lambdan$;
    \item\label{2} $\eta_Y(w)=\dv(w)v-\frac{r}{2}\delta\in\Lambdan$, where $v\in\Lambda_{K3}$ is primitive.
    \item\label{3} $\Pic(S)\cong \Z h$, where $h:=\eta_S^{-1}(v)$.
\end{enumerate}
By \ref{b} and \ref{2} the isotropic Mukai vector 
\begin{equation}\label{Mukai_odd}
    u:= \left(r,\;\dv(w)h,\;\frac{\dv(w)^2h^2}{2r}\right) \in \H^{\ast}(S,\ZZ)
\end{equation} 
is well-defined.
Moreover, by \ref{a} and \ref{c} we have
$$\gcd(r,\;\dv(w))=1=\gcd\left(r,\;\frac{\dv(w)^2h^2}{2r}\right),$$
in particular $u$ is primitive.
Hence, we can perform Construction~\ref{example1}, which yields a projectively hyperholomorphic vector bundle $\mathcal{U}^{[n]} \in \Coh(M^{[n]}\times S^{[n]})$, where $M:= M_S(u)$ is another $K3$ surface. Without loss of generality, in the construction we can choose $\mathcal{U}^{[n]}\in\Coh(M^{[n]}\times S^{[n]})$ as in Remark~\ref{rmk_nicotranquillo}. Moreover, there exist a marking $\eta_{M^{[n]}}$ and a rational isometry $\psin \colon \Lambda_{K3^{[n]},\mathbb{Q}} \to \Lambda_{K3^{[n]},\mathbb{Q}}$ such that $\eta_{S^{[n]}} \circ \psin \circ \eta_{M^{[n]}}$ is a rational Hodge isometry sending some K\"ahler class to a K\"ahler class (see Section \ref{sec_vanishing_twist}). We denote by $\mathfrak{M}^\circ_{\psin}$ the connected component of the moduli space of marked pairs containing $(M^{[n]},\eta_{M^{[n]}},S^{[n]},\eta_{S^{[n]}})$.

Since $(Y,\eta_Y)$ and $(S^{[n]},\eta_{S^{[n]}})$ both lie in $\mathfrak{M}^\circ_\Lambdan$, by Lemma~\ref{lem_pisurj} there exists $(X,\eta_X,Y,\eta_Y) \in \mathfrak{M}^\circ_{\psin}$. As a consequence, by Theorem~\ref{thm_compatiblevectorbundle}, there exist a split K\"ahler class $\omega_{X\times Y}$ and a projectively $\omega_{X\times Y}$-stable hyperholomorphic vector bundle $(\kE,\alpha_\kE)$ on $X\times Y$, which is the kernel of a twisted Fourier--Mukai equivalence
\[
\phi_{\mathcal{E}}\colon \D(X,\alpha_X^{-1}) \to \D(Y,\alpha_Y),
\]
where $\alpha_\kE = \alpha_X + \alpha_Y \in \Br(X\times Y) = \Br(X) \oplus \Br(Y)$.

Since $\eta_Y^{-1}\circ\eta_{S^{[n]}}\cc_1\left(\mathcal{U}^{[n]}|_{\{x\}\times Y}\right)\in\QQ w\subset\Pic_{\QQ}(Y)$, by Lemma~\ref{lem_sufficientcondition}, we have $\alpha_X=0$ and $\alpha_Y=0$.
By Theorem~\ref{thm_compatiblevectorbundle}~(3), together with Remark~\ref{rmk_trasp_c1}, we conclude that the fibers, up to tensoring with a holomorphic line bundle $L\in\Pic(Y)$ satisfy
\begin{equation*}
    \cc_1(\mathcal{E}_x)= 
           r^{n-1}n! w 
\end{equation*}
The computation of $\Delta(\mathcal{E}_x)$ follows from the fact that $\kappa(\mathcal{E})$ is obtained from parallel transport from $\kappa(\mathcal{U}^{[n]})$ and Equations~\eqref{rank1},\eqref{c11},\eqref{eq_modularity1}. 
As a consequence, for all $x \in X$, the bundle $\kE|_{\{x\}\times Y}$ is $\omega_Y$-slope stable. Indeed, by Lemma~\ref{lem_fibersarestable}, the restriction $\kE|_{\{x\}\times Y}$ is $\omega_Y$-slope polystable; moreover, it is simple by \cite[Theorem~3.2.1]{caldararu-thesis}. By Lemma~\ref{lem_omegaimpliesH}, there exists an ample divisor $H \in \Pic(Y)$ such that $\kE|_{\{x\}\times Y}$ is $H$-slope stable, for all $x\in X$. We conclude by Lemma~\ref{lem_Ymodulisp} that $X$ can be identified with a connected component of the moduli space of Gieseker $H$-semistable sheaves.
\end{proof}
\begin{proof}[Proof of Theorem~\ref{thm_main} for $\dv(w)$ even]
The idea of the proof is the same as in the previous case; we explain here the differences. By \ref{a2} and \ref{b2} we have $\ell(w) = [r]$ and $w^2 > -r^2(2n-2)$, so that we can apply Lemma~\ref{markingexistence} with $m = r$. Hence, there exist a marked Hilbert scheme $(S^{[n]},\eta_{S^{[n]}})$ and a marking $\eta_Y \colon \H^2(Y,\mathbb{Z}) \to \Lambdan$ such that
\begin{enumerate}[label=(\arabic*')]
    \item $(Y,\eta_Y)$ and $(S^{[n]},\eta_{S^{[n]}})$ lie in the same connected component of $\mathfrak{M}_\Lambdan$, which we denote by $\mathfrak{M}^\circ_\Lambdan$;
    \item \label{2'} $\eta_Y(w) = \dv(w)v - r\delta \in \Lambdan$, where $v \in \Lambda_{K3}$ is primitive;
    \item $\Pic(S) \cong \Z h$, where $h := \eta_S^{-1}(v)$.
\end{enumerate}
As in the previous case, we can perform Construction~\ref{example1}, but this time we choose
\begin{equation}\label{Mukai_even}
    u := \left(r,\;\frac{\dv(w)}{2}h,\;\frac{\dv(w)^2 h^2}{8r}\right) \in \H^{\ast}(S,\ZZ),
\end{equation}
which can be shown to be well-defined, primitive, isotropic, and such that
\[
\gcd\left(r,\; \frac{\dv(w)}{2}\right) = 1= \gcd\left(r,\; \frac{\dv(w)^2 h^2}{8r}\right),
\]
using \ref{2'}, \ref{a2}, \ref{b2} and \ref{c2}. The rest of the proof is identical to the previous case.
\end{proof}

\begin{rmkn}
    The vector bundles parametrised by the moduli spaces in Theorem~\ref{thm_main} are atomic (see \cite{beckmann}). Indeed they are the image through a derived equivalence of skyscraper sheaves of  points, which are atomic, and atomicity is preserved under derived equivalence.
\end{rmkn}

\subsection{Corollaries of the main result}
In this Section, we apply results of Markman~\cite{markman}, Taelman~\cite{tae23} and Huybrechts--Stellari~\cite{paolo_huy} to compare the Hodge structures of the moduli spaces involved in Theorem \ref{thm_main}.

We recall that the rational cohomology of any smooth projective variety $Z$ is naturally equipped with a lattice structure, given by the Mukai pairing, and with the following weight $0$ Hodge structure \begin{equation}
    \label{HSuntwisted} H^{-k,k}(Z):= \bigoplus_{p-q=k} H^{p,q}(Z),
\end{equation} for any $k= -\dim(Z),\dots, \dim(Z)$ (see \cite[Section 6.1]{Huy03}).

From Theorem \ref{thm_main}, combined with \cite[Proposition 5.39, 5.44]{Huy03}, we deduce the following.
\begin{cor}\label{cor_1}
	In the setup of Theorem \ref{thm_main}, 
 the cohomological Fourier--Mukai transform 
	\begin{align*}
		\tau_{\kE}\colon \H^{\ast}(X,\QQ)&\longrightarrow \H^{\ast}(Y,\QQ)\\
        \alpha &\longmapsto p_{Y,*}\left(p_X^*\alpha\cdot\ch(\kE)\sqrt{\td_{X\times Y}}\right)
	\end{align*}
    is an isometry with respect to the Mukai pairing and an isomorphism of weight $0$ Hodge structures, with respect to (\ref{HSuntwisted}).
\end{cor}

Next, we deal with the twisted case.
Let $Z$ be a smooth projective variety over the complex numbers such that $\H^3(Z,\ZZ)_{\mathrm{tors}}=0$. Following Huybrechts--Stellari \cite{paolo_huy}, we recall the construction of twisted Chern characters for twisted coherent sheaves on $Z$.
Let $\alpha\in \Br(Z)=\H^2_{\mathrm{an}}(Z,\kO_Z^\times)_{\mathrm{tors}}$ and let $B_Z\in\H^2(Z,\QQ)$ be a $B$-field lift of $\alpha$, i.e. a class such that $\exp(B_Z^{0,2})=\alpha$. 
Let $\{\alpha_{ijk}\}$ be a cocycle representing $\alpha$ with respect to a sufficiently fine analytic open cover $\{U_i\}$ of $Z$.
Since the sheaf $\mathcal{C}^{\infty}$ of smooth functions is acyclic, there exist smooth functions $b_{ij}\in \Gamma(U_i\cap U_j, \mathcal{C}^{\infty})$ such that $\alpha_{ijk}=\exp{(-b_{ij}+b_{ik}-b_{jk})}$.
If $E$ is an $\alpha$-twisted coherent sheaf with transition isomorphisms $\phi_{ij}:{E_i}_{|U_i\cap U_j}\to {E_j}_{|U_{i}\cap U_j}$, the modified transition maps $\phi'_{ij}=\exp(b_{ij})\phi_{ij}$
define an untwisted sheaf $E'=(E_{i},\phi'_{ij})$. Following Huybrechts--Stellari, one defines $\ch^{B_Z}(E):=\ch(E')$ (see \cite[Proposition 1.2]{paolo_huy}).

We can generalise the weight $0$ Hodge structure defined in (\ref{HSuntwisted}) to the twisted case as follows:
\begin{equation}\label{eq_defzeroHstru}
    \H^{-k,k,B_Z}(Z):=\exp(B_Z)\left(\bigoplus_{p-q=k}\H^{p,q}(Z)\right).
\end{equation}
Accordingly, we can then state the following.

\begin{prop}[{\cite[Section 4, 4) and 5)]{paolo_huy}}]\label{huystellari}
    A twisted derived equivalence $\D(X,\alpha_X^{-1})\to \D(Y,\alpha_Y)$ with Fourier-Mukai kernel $(\kE,\pr_X^*\alpha_X+\pr_Y^*{\alpha_Y})$ with $\alpha_X\in \Br(X)$ and $\alpha_Y\in\Br(Y)$, induces an isomorphism
    \[
    \tau_{\kE}\colon H^*(X,\QQ)\to H^*(Y,\QQ)
    \]
    defined by the class $\ch^{-B_X+B_Y}(\kE)\sqrt{\td_{X\times Y}}$, where $B_X$ (resp. $B_Y$) is a $B$-field for the Brauer class $\alpha_X$ (resp. $\alpha_Y$). 
    This isomorphism is an isometry with respect to the Mukai pairing and an isomorphism of weight $0$ Hodge structures, with respect to ~\eqref{eq_defzeroHstru}.
\end{prop}

Proposition~\ref{huystellari} naturally applies in the setting of Theorem~\ref{thm_compatiblevectorbundle}, since we have a twisted derived equivalence $\D(X,\alpha_X^{-1})\rightarrow\D(Y,\alpha_Y)$ and explicit $B$-fields for $\alpha_X$ and $\alpha_Y$.

\begin{rmk}\label{rmk_iso_nonfine}
    In the setup of Remark~\ref{rmk_nonfinemoduli}, there is an isomorphism
    \[
    \tau_{\kE}\colon H^*(X,\QQ)\to H^*(Y,\QQ)
    \]
    induced by the class $\ch^{-\pr_X^*B_X}(\kE)$ where $B_X$ is a $B$-field for the Brauer class $\alpha_X$. This isomorphism preserves the Mukai pairing and suitable Hodge structures, as in Proposition~\ref{huystellari}.
\end{rmk}

In the setup of Theorem~\ref{thm_compatiblevectorbundle}, building on the work of \cite{tae23}, Markman constructs a rational Hodge isometry $H^2(X,\mathbb{Q})\rightarrow H^2(Y,\mathbb{Q})$. In the next result, we consider the pure Hodge structure on the second cohomology group of a IHS manifold and the Beauville--Bogomolov--Fujiki form.

\begin{cor}[{\cite[Lemma 4.1 and Theorem 1.1]{markman}}]\label{cor_2}
	In the setup of Theorem~\ref{thm_main} and Theorem~\ref{thm_compatiblevectorbundle} we have a rational isometry and isomorphism of weight $2$ Hodge structures
	\begin{equation}
		 \rathodge{\kE}\colon\H^2(X,\mathbb{Q})\to \H^2(Y,\mathbb{Q})
	\end{equation}
	which is a rational multiple of the following composition
	\begin{equation}
		\H^2(X,\QQ)\xrightarrow{\cup \cc_2(X)^{n-1}}\H^{4n-2}(X,\QQ)\overset{\phi_{\kE}}{\longrightarrow}\H^2(Y,\QQ)
		\end{equation}
		where $\phi_{\kE}$ the correspondence defined by the class $\kappa(\kE)\sqrt{\td_{X\times Y}}$ (see \eqref{eq_def_kappa}).
\end{cor}

Finally, we prove that the component  $X$ of the moduli space in Theorem~\ref{thm_main} comes naturally equipped with a primitive class in $H^2(X,\Z)$.
In what follows, we use the notation introduced in Subsection \ref{sec_vanishing_twist}.
\begin{cor}\label{integral_primitive}
In the setup of Theorem \ref{thm_main}, the class $\rathodge{\kE}^{-1}(w)\in H^2(X,\QQ)$ is integral and primitive.
\end{cor}
\begin{proof}
We treat only the case where $\dv(w)$ is odd, since the other case is similar.

In this case $r$ is even. By \ref{2} and \ref{3}, we obtain
\[
    \rathodge{\U^{[n]}}^{-1}\circ\eta_{S^{[n]}}^{-1}\circ\eta_Y(w)=\rathodge{\U^{[n]}}^{-1}\left(\dv(w)h-\dfrac{r}{2}\delta_{S^{[n]}}\right)=\dv(w)\widehat{h}-\dfrac{r}{2}\delta_{M^{[n]}}\in H^2(M^{[n]},\Z)
\]
The class $\dv(w)\widehat{h}-\dfrac{r}{2}\delta_{M^{[n]}}$ is primitive, since $\gcd(\dv(w),r)=1$. By \eqref{commutation_markings}, it follows that
\[
    \rathodge{\kE}^{-1}(w)=\eta_X^{-1}\circ\eta_{M^{[n]}}\circ\rathodge{\U^{[n]}}^{-1}\circ\eta_{S^{[n]}}^{-1}\circ\eta_Y(w)=\eta_X^{-1}\circ\eta_{M^{[n]}}\left(\dv(w)\widehat{h}-\dfrac{r}{2}\delta_{M^{[n]}}\right).
\]
Therefore $\rathodge{\kE}^{-1}(w)$ is integral and primitive.
\end{proof}

\section{Examples of moduli spaces}\label{sec_examples}
In this Section, we give examples to which Theorem~\ref{thm_main} applies. More precisely, for every $n$ there are infinitely many polarisation types such that every polarised IHS $(Y,w)$ of $K3^{[n]}$-type, with $w$ of one of those types, satisfies the assumption of Theorem \ref{thm_main}. This yields examples of smooth components of moduli spaces of modular vector bundles on IHS varieties of every even dimension.

Let $^{n}\mathcal{M}_m^{(\gamma)}$ denote the coarse moduli space of polarised IHS manifolds $(Y,w)$ of $K3^{[n]}$-type such that $w$ is ample, of divisibility $\dv(w)=\gamma$, and $w^2=2m$ (see \cite[Theorem 1.5]{GHS_moduli_IHS} for its construction). This moduli space is quasi-projective. Moreover, by \cite{apost_moduli_red}, if either $\gamma=1$, or $\gamma=2$ and $n+m\equiv 1\pmod{4}$, then $^{n}\mathcal{M}_m^{(\gamma)}$ is irreducible of dimension 20 (see also \cite[Theorem 3.5]{olivier_HK}). By the works of O'Grady, Beauville--Donagi, Debarre--Voisin and Iliev--Ranestad, the general element of $^{2}\mathcal{M}_m^{(\gamma)}$ admits an explicit construction for $(\gamma,m)=(1,1), (2,3), (2,11), (2,19)$. We refer to \cite[§3.6.1]{olivier_HK} for an overview. In what follows, we verify that these four cases are covered by Theorem~\ref{thm_main} (see Corollaries~\ref{cor_EPWsextic} and~\ref{cor_BD_DV_IR}). On the other hand, Theorem~\ref{thm_main} does not apply to the LLSvS eightfolds or to the EPW cube sixfolds (see \cite{LLSvS} and \cite{EPWcubes}).

\subsection{Divisibility one}
Let us fix a polarised IHS variety $(Y,w)\in {}^{n}\mathcal{M}_{m}^{(1)}$ with polarisation of divisibility one. 
This is irreducible of dimension 20.
Let us fix $r'\in \ZZ_{\geq 1}$ such that
\begin{itemize}
    \item $k:=\dfrac{w^2}{4r'}+\dfrac{r'(n-1)}{2}\in \ZZ_{\geq 1}$
    \item $(k,2r')=1$.
\end{itemize}
It is then easy to observe that choosing $r:=2r'$ we can always verify the assumptions of Theorem \ref{thm_main}.
Solving in $w^2$ we get
\[
w^2=2r'(2k-r'(n-1)).
\]
Therefore, starting from $r'$, we can choose infinitely many $k$ and $w$ satisfying those assumptions.
\begin{cor}\label{cor_div1}
    Let  $r',k\in \ZZ_{\geq 1}$ such that $(k,2r')=1$ and $2r'(2k-r'(n-1))\geq 1$.
    Then, for any $(Y,w)\in {}^{n}\mathcal{M}_{w^2/2}^{(1)}$ with $w^2=2r'(2k-r'(n-1))$ there is a polarisation $H$ on $Y$ and a connected component $X$ of the moduli space $M_{Y,H}$ of Gieseker $H$-semistable sheaves on $Y$ such that
    \begin{itemize}
        \item $X$ parametrises $\mu_H$-stable modular vector bundles $E$ with 
        \[rk(E)=(2r')^n n! \qquad c_1(E)=(2r')^{n-1}n!w \qquad \Delta(E)=\frac{((2r')^n n!)^2}{12}\cc_2(Y)\]
        \item $X$ and $Y$ are derived equivalent.
    \end{itemize}
    Moreover, if $\Pic(Y)\cong \Z w$ then $(X,\rathodge{\kE}^{-1}(w))\in {}^{n}\mathcal{M}_{w^2/2}^{(1)}$.
\end{cor}
\begin{proof}
    For the first two claims, it is enough to notice that, given $r',k\in \ZZ_{\geq 1}$ such that $(k,2r')=1$ and $2r'(2k-r'(n-1))\geq 1$, the assumptions of Theorem \ref{thm_main} are satisfied.
    
     Notice that $\rathodge{\kE}^{-1}(w)$ is integral and primitive, due to Corollary \ref{integral_primitive}, is Hodge and $\rathodge{\kE}^{-1}(w)^2=w^2$ by Corollary \ref{cor_2}.
    If $\Pic(Y)\cong \Z w$, then we deduce that $\Pic(X)\cong \Z\rathodge{\kE}^{-1}(w)$, so we are only left to prove that this class is ample.
    Recall that $\rathodge{\kE}^{-1}$ sends some K\"ahler class on $Y$ to some K\"ahler class on $X$, hence, it not only maps the positive cone in $H^2(Y,\QQ)$ to the positive cone of $H^2(X,\QQ)$ but also preserves the connected component containing K\"ahler classes.
    Therefore, the class $\rathodge{\kE}^{-1}(w)$ is not anti-ample, hence is ample.
\end{proof}

The examples with $r=2r'\leq 4$ and $k\leq 5$ are listed as follows: 
$$\begin{array}{cc|c}
        r&k&w^2\\
        \hline
        2&1&2(3-n)\\
        2&3&2(7-n)\\
        2&5&2(11-n)\\
        4&1&8(2-n)\\
        4&3&8(4-n)\\
        4&5&8(6-n),
    \end{array}$$
note that we consider only $n\in \ZZ_{\geq 1}$ such that $w^2>0$.
If $n=2$, the numerics on the first line give us the EPW sextic and Theorem \ref{thm_main} together with Corollary \ref{cor_div1} reads as follows.
\begin{cor}[EPW sextic]\label{cor_EPWsextic}
    Let $(Y,w)\in {}^{2}\mathcal{M}_{1}^{(1)}$.
    Then, there is an ample divisor $H$ on $Y$ and a connected component $X$ of the moduli space of $H$-Gieseker semistable sheaves such that $X$ is of $K3^{[2]}$-type, derived equivalent to $Y$ and parametrises $H$-slope stable modular vector bundles $E$ with
    \[
    \rk(E)=8, \qquad c_1(E)=4w \qquad \Delta(E)=\frac{16}{3}\cc_2(X).
    \]
    Moreover, if $\Pic(Y)\cong \Z w$ then $X\in {}^{2}\mathcal{M}_{1}^{(1)}$.
\end{cor}
This example was already discussed in \cite[Proposition 5.2]{kapustka-kapustka-derived-equivalent-hk4}, where it is shown that the connected component of the moduli space $X$ can be taken to be the dual EPW sextic.
More precisely, in \cite[Section 4.2]{OG15}, O'Grady shows that the correspondence between an EPW sextic and its dual gives an automorphism of ${}^{2}\mathcal{M}_{1}^{(1)}$.

\subsection{Divisibility two}
Let us fix a polarised IHS variety $(Y,w)\in {}^{n}\mathcal{M}_{m}^{(2)}$ with polarisation of divisibility two. 
This is non empty, irreducible of dimension 20 whenever $n+m\equiv 1\pmod{4}$.

\begin{cor}\label{cor_div2}
    Let $(Y,w)\in {}^{n}\mathcal{M}_{m}^{(2)}$ with $n+m\equiv 1\pmod{4}$ and let us consider an odd integer $r\in \ZZ_{\geq 3}$ such that $k:=w^2/(2r)\in \ZZ$ and $\gcd(r,k)=1$.
    Then there there is a polarisation $H$ on $Y$ and a connected component $X$ of the moduli space $M_{Y,H}$ of Gieseker $H$-semistable sheaves on $Y$, such that
     \begin{itemize}
        \item $X$ parametrises $\mu_H$-stable modular vector bundles $E$ with 
        \[rk(E)=r^nn! \qquad c_1(E)=r^{n-1}n!\dfrac{w}{2} \qquad \Delta(E)=\frac{(r^n n!)^2}{12}\cc_2(Y)\]
        \item $X$ and $Y$ are derived equivalent.
    \end{itemize}
    Moreover, if $\Pic(Y)\cong \Z w$ then $(X,\rathodge{\kE}^{-1}(w))\in {}^{n}\mathcal{M}_{m}^{(2)}$.
\end{cor}
The proof follows the same scheme as the one of Corollary \ref{cor_div1}.

Some examples are listed below:
$$\begin{array}{cc|c}
        r&k&w^2\\
        \hline
        3&1,2,4,5,\dots&6,12,24,30,\dots\\
        5&1,2,3,\dots &10,20,30,\dots\\
        7&1,2,3,\dots&14,28,42,\dots\\
        11&1,\dots&22,\dots\\
        19&1,\dots&38,\dots\\
    \end{array}$$
Let us observe that for any 
$n\geq 1$ and $n\equiv 0,1,3\pmod{4}$ we have examples with $r=3, k=1,2,4$.
Indeed, the condition $n+m\equiv 1\pmod{4}$ becomes $n+rk\equiv 1\pmod{4}$, i.e. $k\equiv n-1 \pmod{4}$.
When $n\equiv 2\pmod{4}$ we can choose $r=5, k=3$.

In the case of fourfolds, we have three explicit constructions of locally complete families to which we can apply Theorem \ref{thm_main}, yielding the following.
\begin{cor}[Beauville--Donagi, Debarre--Voisin, Iliev--Ranestad]\label{cor_BD_DV_IR} Let $(Y,w)$ be a polarised IHS manifold  of $K3^{[2]}$-type such that $\dv(w)=2$ and one of the following holds:
\begin{itemize}
    \item [{\normalfont{(a)}}] $w^2=6$,
    \item [{\normalfont{(b)}}] $w^2=22$,
    \item [{\normalfont{(c)}}] $w^2=38$.
\end{itemize}
Then there is an ample divisor $H$ on $Y$ and a connected component $X$ of the moduli space of Gieseker $H$-semistable sheaves such that $X$ is of $K3^{[2]}$-type, derived equivalent to $Y$ and parametrising $H$-slope stable vector bundles $E$ with
\begin{itemize}
    \item [{\normalfont{(a')}}] $rk(E)=18,\quad c_1(E)=3h \quad \Delta(E)=27\cc_2(Y)$ in case $(a)$,
    \item [{\normalfont{(b')}}] $rk(E)=2\cdot11^2,\quad c_1(E)=11h\quad \Delta(E)=\frac{11^4}{3}\cc_2(Y)$ in case $(b)$,
    \item [{\normalfont{(c')}}] $rk(E)=2\cdot19^2,\quad c_1(E)=19h,\quad \Delta(E)=\frac{19^4}{3}\cc_2(Y)$ in case $(c)$.
\end{itemize}
Moreover, if $\Pic(Y)\cong \Z w$ then $X\in {}^{2}\mathcal{M}_{w^2/2n}^{(2)}$.
\end{cor}

\appendix

\section{Hyperholomorphic vector bundles} \label{app_hyperholo}

Let $(X,\omega_X)$ and $(Y,\omega_Y)$ be IHS manifolds with fixed K\"ahler classes $\omega_X,\omega_Y$. Choose hyperk\"ahler structures on $X$ and $Y$ that are compatible with the fixed K\"ahler classes.
Finally, equip $X\times Y$ with the diagonal hyperk\"ahler structure, which is compatible with $\omega_{X\times Y} := \pi_X^*\omega_X + \pi_Y^*\omega_Y$.
The aim of this Appendix is to prove the following fact.

\begin{lem}\label{lem_fibersarestable}
Let $(\mathcal{E},\alpha_{\mathcal{E}})$ be a projectively $\omega_{X\times Y}$-stable hyperholomorphic vector bundle. 
Then for all $x \in X$, the fiber $\kE|_{\{x\}\times Y}$ is $\omega_Y$-slope polystable.
\end{lem}

Together with the simplicity of the fibers, it allows us to prove that the projectively hyperholomorphic vector bundle $\kE \in \Coh(X\times Y,\alpha_\kE)$ constructed in Theorem~\ref{thm_compatiblevectorbundle} gives a family of stable vector bundles on $Y$, parametrised by $X$.
For this purpose, we recall the theory of (projectively) hyperholomorphic sheaves following Verbitsky (see \cite{verbitsky-hyperholomorphic-sheaves} for a deeper discussion). 
We refer the reader to \cite{markman_bbf} for a generalisation of Verbitsky's results and to \cite{bottini-macri-stellari-hk-varieties} for an overview on the topic.

Before proving Lemma~\ref{lem_fibersarestable}, we recall some basic definitions. 
Let $(M,g)$ be a compact Riemannian manifold endowed with a hyperk\"ahler structure (see \cite[Definition~2.1]{verbitsky-hyperholomorphic-sheaves}). Thus, there exist complex structures $I,J,K\in\End(TM)$, parallel with respect to $g$, satisfying the quaternionic relations
\[
I^2=J^2=K^2=IJK=-\id_{TM}.
\]
These endomorphisms induce an action of the quaternion algebra $\mathbb{H}$ on the tangent bundle $TM$.
In particular, the bundle $TM$ is acted on by the group of unit quaternions, which is isomorphic to $SU(2)$. 
Now, for every
\[
L=aI+bJ+cK, \qquad a,b,c\in\mathbb{R}, \qquad a^2+b^2+c^2=1,
\]
we have $L^2=-\id_{TM}$, so $L$ defines a complex structure on $M$, and the metric $g$ is K\"ahler with respect to $L$. 
We denote by $\omega_L$ the associated K\"ahler form. The set of induced complex structures is naturally identified with the real sphere $S^2$.

\begin{rmk}
The action of $SU(2)$ commutes with the Laplacian operator, hence induces an action on cohomology (see \cite[Lemma~2.5]{verbitsky-hyperholomorphic-sheaves}). Moreover, a differential $2p$-form is $SU(2)$-invariant if and only if it is of Hodge type $(p,p)$ with respect to every induced complex structure (see \cite[Lemma~2.6]{verbitsky-hyperholomorphic-sheaves}).
\end{rmk}
\begin{df} Let $E$ be a holomorphic vector bundle on the complex manifold $(M,L)$. \begin{enumerate}
    \item $E$ is called $\omega_{L}$-\textit{slope stable hyperholomorphic} if it is $\omega_L$-slope stable and $c_i(E)\in \H^{2i}(M,\mathbb{R})$ is $SU(2)$-invariant for $i=1,2$. 
    \item  $E$ is called $\omega_L$-\textit{slope polystable hyperholomorphic} if it is $\omega_L$-slope polystable and each of its stable summands is $\omega_L$-slope stable hyperholomorphic.
\end{enumerate}
\end{df}

We will need the following result.

\begin{thm}\label{thm_hyperholocharacterization}
    A holomorphic vector bundle $E$ on $(M,L)$ is $\omega_L$-slope polystable hyperholomorphic if and only if it admits a hyperholomorphic connection, i.e.~$E$ admits a hermitian metric such that the curvature of the Chern connection is $SU(2)$-invariant.
\end{thm}
\begin{proof}
    See \cite[Theorem 3.19]{verbitsky-hyperholomorphic-sheaves}.
\end{proof}

\begin{df}\label{def_proj_hyperhol}
    A (possibly twisted) holomorphic vector bundle $E$ on $(M,L)$ is \textit{projectively $\omega_{L}$-slope stable hyperholomorphic} if it is $\omega_L$-slope stable and $\cEnd(E)$ is $\omega_L$-polystable hyperholomorphic.
\end{df}

If $h$ is a hermitian metric on a holomorphic vector bundle on a compact complex manifold, we denote by $\nabla_h$ the induced Chern connection and by $F_h$ the associated curvature.
With this notation, we can make the following remark; see \cite[Section 1.5]{kobayashi}.

\begin{rmk}\label{rmk_pullbackmetric}
    Let $f\colon X\rightarrow Y$ be a holomorphic map between complex manifolds.
    Let $\mathcal{F}$ be a holomorphic vector bundle on $Y$ carrying a compatible hermitian metric $h$. Then, the induced hermitian metric $f^\ast h$ on $f^\ast\mathcal{F}$ satisfies the following relations:
    \begin{equation}
        \nabla_{f^\ast h}=f^\ast\nabla_h, \quad F_{f^\ast h}=f^\ast F_{h}.
    \end{equation}
\end{rmk}

Motivated by Remark~\ref{rmk_pullbackmetric}, we easily obtain the following lemma:
\begin{lem}\label{lem_invariance}
\normalfont{(1)} In the setup of Remark~\ref{rmk_pullbackmetric}, assume furthermore that a Lie group $G$ acts smoothly by automorphisms of the real tangent bundles $T_X$ and $T_Y$, and that $df\colon T_X\rightarrow f^\ast T_Y$ is a $G$-equivariant map.
If $F_h$ is $G$-invariant then $F_{f^\ast h}$ is $G$-invariant as well.
    
\noindent \normalfont{(2)} Let $X\times Y$ be a product of hyperk\"ahler manifolds carrying a holomorphic vector bundle $\mathcal{F}$, with a compatible hermitian metric $h$. Let $x\in X$ be a point and denote by $i\colon \{x\}\times Y\rightarrow X\times Y$ the inclusion of the fiber. If the curvature $F_h$ is $SU(2)$-invariant with respect to the diagonal $SU(2)$-action on $T_{X\times Y}$, then $F_{i^\ast h}$ is $SU(2)$-invariant with respect to the action on $T_Y$.
\end{lem}
\begin{proof}
For (1), let $p\in X$, let $u,v\in T_p X$ and let $g\in G$. We have
\begin{align*}
(g\cdot F_{f^*h})(u,v)
&= F_{f^*h}(g^{-1}\cdot u,g^{-1}\cdot v) 
&& \text{by definition of the action of $G$ on $\mathcal{A}^2(X,\cEnd(f^\ast\mathcal{F}))$} \\
&= (f^\ast F_h)(g^{-1}\cdot u,g^{-1}\cdot v)
&& \text{by Remark \ref{rmk_pullbackmetric}}\\
&= F_h(df(g^{-1}\cdot u),df(g^{-1}\cdot v))
&& \text{by definition of $f^\ast\colon\mathcal{A}^2(Y,\cEnd(\mathcal{F}))\rightarrow\mathcal{A}^2(X,\cEnd(f^\ast\mathcal{F}))$} \\
&= F_h(g^{-1}\cdot df(u),g^{-1}\cdot df(v))
&& \text{by $G$-equivariance of $df$} \\
&= F_h(df(u),df(v))
&& \text{by $G$-invariance of $F_h$} \\
&= F_{f^*h}(u,v)
&& 
\end{align*}
This is enough to conclude (1).

Part (2) follows directly from Part (1), taking $G=SU(2)$ and $f=i$.
\end{proof}

The last ingredient for the proof of Lemma~\ref{lem_fibersarestable} is the following remark.

\begin{rmk}\label{rmk_polyiffendpoly}
Let $(Y,\omega_Y)$ be a compact K\"ahler manifold with a fixed K\"ahler class, and let $E$ be a (possibly twisted) holomorphic vector bundle. Then $E$ is $\omega_Y$-polystable if and only if $\cEnd(E)$ is $\omega_Y$-polystable. In fact, this follows from \cite[Corollary~3.8]{anchouche-biswas-einstein-hermitian}; see also \cite[Theorem 2.17]{OG26c}.
\end{rmk}

\begin{proof}[Proof of Lemma \ref{lem_fibersarestable}]
We denote by $i\colon \{x\}\times Y\rightarrow X\times Y$ the inclusion of the fiber. By assumption, the untwisted vector bundle $\cEnd(\kE)$ is $\omega_{X\times Y}$-polystable and hyperholomorphic. Hence, by Theorem~\ref{thm_hyperholocharacterization}, it admits a hermitian metric $h$ inducing a hyperholomorphic connection $\nabla_h$, whose curvature is $SU(2)$-invariant with respect to the diagonal action of $SU(2)$ on $T_{X\times Y}$. By Remark~\ref{rmk_pullbackmetric}, the restriction $i^\ast h$ on $i^\ast(\cEnd(\kE))=\cEnd(i^\ast\kE)$ is a hermitian metric such that the curvature of the induced Chern connection is given by $i^\ast F_{h}$. By Lemma~\ref{lem_invariance}, the curvature $F_{i^\ast h}$ is $SU(2)$-invariant, since $F_{h}$ is. Consequently, since $\nabla_{i^\ast h}$ is a hyperholomorphic connection on $\cEnd(i^\ast\kE)$, we deduce that $\cEnd(i^\ast\kE)$ is $\omega_Y$-polystable, by Theorem~\ref{thm_hyperholocharacterization}. The result then follows from Remark~\ref{rmk_polyiffendpoly}.
\end{proof}

\begin{rmk}\label{twistedmoduli}
    In the proof of Theorem~\ref{thm_main}, we apply Lemma~\ref{lem_fibersarestable} to a vector bundle $(\mathcal{E},\alpha_{\mathcal{E}})$ such that $\alpha_Y = 0$, where we use the decomposition $\alpha_\kE = \alpha_X + \alpha_Y \in \Br(X\times Y) = \Br(X) \oplus \Br(Y)$.
    The fact that Lemma~\ref{lem_fibersarestable} does not need $\alpha_Y=0$ suggests that the same construction in Theorem~\ref{thm_main} can also be used to exhibit components of moduli spaces of twisted sheaves which are IHS manifolds of $K3^{[n]}$-type.
\end{rmk}

\section{Variation of slope stability}\label{appslopestab}
In \cite[Section 5]{OGrady2026}, O'Grady constructs a wall-and-chamber decomposition of the K\"ahler cone of an IHS manifold in order to study variation of slope stability for modular sheaves. In the present paper, we need an analogous statement, valid without the modularity assumption: for a fixed family of torsion-free sheaves, the set of K\"ahler classes with respect to which all fibres are slope-stable is open (see Lemma~\ref{lem_stability_open}). Moreover, if this open subset is nonempty, then it contains an ample class (see Lemma~\ref{lem_omegaimpliesH}). 
We include short proofs of these two facts in order to keep the paper self-contained, following arguments along the lines of \cite{OGrady2026,greb_ross_toma_moduli}.

Let $Y$ be a projective IHS manifold. For $E,F\in\Coh(Y)$, we let
\[
    \lambda_{F,E}:=\rk(E)c_1(F)-\rk(F)c_1(E)
\]
Let $\omega\in \H^{1,1}(Y,\mathbb{R})$ be a K\"ahler class. A torsion-free coherent sheaf is $\omega$-slope stable if and only if for all coherent subsheaves $0\subsetneq F\subsetneq E$ such that $0<\rk(F)<\rk(E)$, we have
\[
b_Y(\lambda_{F,E},\omega)<0,
\]
where $b_Y$ is the BBF bilinear form (see \cite[Lemma 5.1]{OGrady2026}). For a K\"ahler class $\omega$, we set
\[
S^{r,\Delta}_\omega := \left\{E\;\;\middle|\;
\begin{array}{c}
\text{$E$ is $\omega$-slope stable torsion-free sheaves} \\
\text{of rank $r$ and discriminant $\Delta$}
\end{array}
\right\}.
\]
\begin{lem}\label{lem_stability_open}
    Let $Y$ be a projective IHS manifold and $\omega\in \H^{1,1}(Y,\RR)$ a K\"{a}hler class. Let $r\in \mathbb{N}_{>0}$ and $\Delta\in \H^4(Y,\Z)$. Then the class $\omega$ admits an open neighbourhood $V\subset\mathcal{K}_Y$ such that for all $\omega'\in V$, we have $S^{r,\Delta}_\omega\subset S^{r,\Delta}_{\omega'}$.
\end{lem}

\begin{proof}
    Let $B\subset \mathcal{K}_Y$ be an open ball with center $\omega$.
    Let $E\in S^{r,\Delta}_\omega$ and suppose it is unstable for some $w'\in B$.
    By \cite[Lemma 6.2]{greb_toma} there is a $\bar{t}\in (0,1]$ such that $E$ is $\bar{t}\omega+(1-\bar{t})\omega' $-slope strictly semistable. 
    Hence, the path $t\omega+(1-t)\omega', t\in [0,1]$  must cross a wall of the form
    \[
    W_{F,E} = \{\alpha\in \mathcal{K}_Y |\;\;b_Y(\lambda_{F,E}, \alpha) = 0\; \}
    \]
    where  $F\subset E$ is maximally $\bar{t}\omega+(1-\bar{t})\omega'$-destabilising (hence saturated).
    
    We now show that only a finite number of such walls intersect $B$. 
    It is enough to show that there are only a finite number of  $\lambda_{F,E}\in \NS(Y)$, where $E\in S^{r,\Delta}_\omega$ and
    $0\neq F\subsetneq E$ is a subsheaf with $\rk F<\rk E$ such that $b_Y(\lambda_{F,E}, \alpha)=0$ for some $ \alpha\in B$.
    Equivalently, picking a norm $\lVert \cdot\rVert$ on $\NS(Y)_\mathbb{R}$, it is enough to prove:
    \[
\sup \left\{ 
    \lVert \lambda_{F,E} \rVert 
    \;\middle|\; 
    \begin{aligned} 
        & E\in S^{r,\Delta}_\omega \\ 
        & E \text{ is } \alpha\text{-slope strictly semistable} \\[-0.5ex]
        & \text{for some } \alpha \in B, \text{ and } F \subset E \\[-0.5ex]
        & \text{is maximally destabilising} 
    \end{aligned} 
\right\} < \infty
\]
In fact, we always have $\lambda_{F,E}\in\NS(Y)$, which is discrete in $\NS(Y)_\mathbb{R}$, and a discrete and bounded subset is finite. 
Let us fix $\alpha, E,F$ be as above and consider $Q:=E/F$.  By \cite[Lemma 3.9]{OG22} we have the following 
\[
\frac{\Delta(E)}{r_E}+\frac{\left(\lambda_{F,E}\right)^2}{r_Er_Fr_Q}=\frac{\Delta(F)}{r_F}+\frac{\Delta(Q)}{r_Q}
\]
It is easy to check that both $F,Q$ are $\alpha$-slope semistable hence, by Bogomolov inequality (see \cite[Theorem 3.2]{pavel_toma_survey}), we obtain
\[
    \frac{\Delta(E)}{r_E}\cdot\alpha^{2n-2}+\frac{\left(\lambda_{F,E}\right)^2}{r_Er_Fr_Q}\cdot\alpha^{2n-2}=\frac{\Delta(F)}{r_F}\cdot\alpha^{2n-2}+\frac{\Delta(Q)}{r_Q}\cdot\alpha^{2n-2}\geq 0
\]
Since $\alpha$ belongs to a relatively compact subset, the quantity $\frac{\Delta(E)}{r_E}\cdot\alpha^{2n-2}$ is bounded from above, hence there is a constant $C$ depending only on $r,\Delta$ and $B$, such that
\[
-\left(\lambda_{F,E}\right)^2\cdot\alpha^{2n-2}< C.
\]
By the Fujiki identity, the positive constant $C':=\frac{(2n-2)!!}{2n-1}c_Y$ yields the inequality
\[
    \left(\lambda_{F,E}\right)^2\cdot\alpha^{2n-2}=C'q_Y(\lambda_{F,E})q_Y(\alpha)^{n-1}.
\]
Since $\overline{B}\subset\mathcal{K}_Y$ is compact, there exists $\epsilon>0$, such that $q_Y(\alpha)>\epsilon$ for all $\alpha\in B$. Combining this with the previous inequality, we obtain that there exists a constant $C''$ only depending on $r,\Delta$ and $B$ such that
\[
-q_Y(\lambda_{F,E})<C''
\]
This immediately implies the claim that $\lVert \lambda_{F,E}\rVert$ is uniformly bounded. In fact the signature of the BBF form tells us that 
\[
m_\alpha=\min\{-q_Y(\beta)\;|\;\beta\in\alpha^\perp,\; \lVert \beta\rVert= 1\}
\]
is positive, it depends continuously on $\alpha$ and $\alpha$ belongs to a relatively compact subset, there exists $C'''>0$ depending only on $B$ such that $m_\alpha>C'''$. 
Note that we can assume $\lambda_{F,E}\neq 0$ hence we have
\[
-q_Y(\lambda_{F,E})=-\lVert \lambda_{F,E}\rVert^2q_Y\left(\frac{\lambda_{F,E}}{\lVert \lambda_{F,E}\rVert}\right)>\lVert \lambda_{F,E}\rVert^2C'''
\]
Combining the inequalities above, we get
\[
\lVert \lambda_{F,E}\rVert^2<\frac{C''}{C'''}
\]
obtaining the required bound. 

We have just showed that  a finite number of walls intersect $B$. 
Hence by choosing $V$ to be a ball with center $\omega$ and sufficiently small radius we obtain the claim.
\end{proof}

An immediate Corollary of the previous Lemma is the following.
\begin{lem}\label{lem_omegaimpliesH}
    Let $\omega\in\mathcal{K}_Y$. Then there exists $H\in\NS(Y)$ ample such that $S^{r,\Delta}_\omega\subset S^{r,\Delta}_{H}$.
\end{lem}
\begin{proof}
We decompose $\omega = x + y \in \H^{1,1}(Y,\mathbb{R}) = \NS(Y)_{\mathbb{R}} \oplus \big(T(Y)_{\mathbb{R}} \cap \H^{1,1}(Y,\mathbb{R})\big)$. 

We will now apply the Huybrechts--Boucksom criterion,  \cite[Théorème 1.2]{boucksom}, to prove that $x$ also lies in $\mathcal{K}_Y$. In fact $q_Y(\omega)=q_Y(x)+q_Y(y)$ and $T(Y)$ is negative definite, hence $q_Y(\omega)>0$ implies that $q_Y(x)>0$. In particular, $x\neq 0$. 
We consider the path $\omega_t=x+ty$.
We observe that $q_Y(\omega_t)>0$ for all $t\in [0,1]$: indeed, 
\[q_Y(\omega_t)=q_Y(x)+t^2q_Y(y)\ge q_Y(x)+q_Y(y)=q_Y(\omega)>0,\]
hence $x$ and $\omega$ are in the same connected component of the cone $\{\alpha\in \H^{1,1}(Y,\RR)|\;q_Y(\alpha)>0\}$. 
By the Huybrechts-Boucksom criterion  it suffices to check that for any rational curve $C\subset Y$ we have $C.x\geq 0$.
Since $y$ is transcendental, we have $C.y=0$ hence $0\leq \omega.C=x.C+y.C=x.C$.

It is straightforward to check that $S^{r,\Delta}_\omega\subset S^{r,\Delta}_{x}$. By Lemma \ref{lem_stability_open}, there exists $x'\in\NS_{\mathbb{Q}}(Y)$ such that $S^{r,\Delta}_{x}\subset S^{r,\Delta}_{x'}$. For an integer $m\gg 0$, we have $H:=mx'\in\NS(Y)$ and it is ample. It follows that $S^{r,\Delta}_{x}\subset S^{r,\Delta}_{x'}=S^{r,\Delta}_{H}$.
\end{proof}

\end{document}